\documentclass{article}
\usepackage[utf8]{inputenc}
\usepackage{amsmath,amsfonts, amsthm,amssymb,bbm,mathtools,hyperref,mathrsfs,stmaryrd,authblk}
\usepackage[a4paper, total={5.8in, 9.5in}]{geometry}

\PassOptionsToPackage{hyphens}{url}
\usepackage{hyperref}

\newtheorem{theorem}{Theorem}[section]
\newtheorem{cor}[theorem]{Corollary}
\newtheorem{lem}[theorem]{Lemma}

\newtheorem{remark}[theorem]{Remark}
\newtheorem{definition}[theorem]{Definition}
\newtheorem{example}[theorem]{Example}
\newtheorem{assumption}{Assumption}[section]

\numberwithin{equation}{section}

\newcommand{\RR}{\mathbb{R}} 
\newcommand{\fH}{\mathcal{H}} %
\newcommand{\EE}{\mathbb{E}} 
\newcommand{\PP}{\mathbb{P}} 
\newcommand{\fM}{\mathcal{M}} 
\newcommand{\fF}{\mathcal{F}} 

\newcommand{\1}{\mathbbm{1}} 
\newcommand{\fP}[1]{\mathcal{P}_{#1}} 
\newcommand{\fL}{\mathcal{L}} 
\newcommand{\fD}{\mathcal{D}} 
\newcommand{\fI}{\mathcal{I}}
\newcommand{\TT}{\mathbb{T}} 
\newcommand{\MM}{\mathbb{M}} 
\newcommand{\fS}{\mathcal{S}} 
\newcommand{\ZZ}{\mathbb{Z}} 

\newcommand{\fG}{\mathcal{G}}
\newcommand{\NN}{\mathbb{N}}
\newcommand{\fN}{\mathcal{N}}

\newcommand{\fK}{\mathcal{K}}

\newcommand{\fB}{\mathcal{B}}

\newcommand{\DD}{\mathbb{D}}
\newcommand{\fE}{\mathcal{E}}
\newcommand{\para}{\olessthan}
\newcommand{\fZ}{\mathcal{Z}}

\title{A Generalized Rough Super Brownian Motion}
\author{Ruhong Jin}
\affil{University of Oxford, \href{mailto:ruhong.jin@maths.ox.ac.uk}{ruhong.jin@maths.ox.ac.uk}}
\date{\today}

\begin{document}
\maketitle

\begin{abstract}
    In this paper, we consider a scaling limit of some branching random walk in random environment whose off-spring distribution has infinite variance. The Laplace functional of the limiting super-process is given by $\EE\left[e^{-\langle\mu_t,\varphi_0\rangle}\right] = e^{-\langle\mu_0,U_t(\varphi_0)\rangle}$, where $U_t(\varphi)$ is the solution to the equation 
    \[
        \partial_t\varphi = \fH\varphi - \varphi^{1+\beta},\qquad \varphi(0) = \varphi_0.
    \]
    The existence and uniqueness of this non-linear equation are also proved as an intermediate step. We also give a martingale characterization of above super-process and show that it possesses the compact support property.
\end{abstract}

\section{Introduction}
The branching random walk in the random environment(BRWRM) is a random walk with stochastic branching mechanism and branching rate. This work explores the large-scale behaviour of some BRWRM whose branching mechanism and rate depend on the random environment at each spatial points.

The scaling limit of branching random walk has been discussed thoroughly in \cite{dawson1993measure}. The aim of this paper is to generalize the results in \cite{dawson1993measure} to the BRWRM with a rough environment.

Given a random potential 
\[
    V(x) = \xi(x), \qquad \{\xi(x)\}_{x\in\ZZ^2}\text{ i.i.d }\sim \Phi,
\]
with $\EE[\Phi] = 0,\EE[\Phi^2] = 1$ on lattice $\ZZ^2$, it is shown in \cite{perkowski2021rough} that certain scaled branching random walks in the random environment with birth(one particle) rate $\xi_+$ and die rate $\xi_-$ will converge to a measure-valued Markov process whose Laplace transformation is associated to the stochastic PDE.
\begin{equation}
    \partial_t\varphi = (\Delta + \xi)\varphi -\varphi^2,\qquad \varphi(0) =\varphi_0,
\end{equation}
where $\xi$ is the white noise.

The BRWRM considered in \cite{perkowski2021rough} has a finite variance off-spring distribution and a general question is what happens if the variance is infinite. It is discussed in \cite{dawson1993measure,etheridge2000introduction} that the Laplace functionals of the limit super processes given by classical branching random walks should satisfy the PDE
\[
    \partial_t\varphi = L\varphi - b(x)\varphi -c(t,x)\varphi^2 + \int_0^\infty(1-e^{-\theta \varphi(t,x)}-\theta\varphi(t,x))n(x,d\theta),
\]
where $A$ is a generator of a Feller semi-group, $b$ is bounded and $c > 0$ and $n$ is some sort of transition kernel. 

Due to the result in \cite{perkowski2021rough}, it is then natural to think that the super-process should exist when $b$ is only a distribution, for example, white noise $\xi$. We do not try to answer this general question in the paper since we are lacking knowledge of dealing with the part with transition kernel $n$. However, we start with a non-trivial but simpler case when $n(x,d\theta) \sim \theta^{\beta +2}$ for $0 < \beta < 1$. In other words, we try to deduce the existence and uniqueness of the super-process whose Laplace functional is given by the SPDE
\begin{equation}\label{equ.beta_PAM}
    \partial_t\varphi = (\Delta + \xi)\varphi -\varphi^{1+\beta}.
\end{equation}
To this goal, we consider the BRWRM whose branching mechanism is given by 
\[
    \frac{2\xi_+}{(1+\beta)|\xi|}(1-s)^{1+\beta} - \frac{\xi}{|\xi|} + \frac{2\xi_+}{|\xi|}s,
\]
and show the scaling limit of this branching random walk gives us exactly the desired super process.

It is worth mentioning that there is actually no general existence and uniqueness theory of equation \eqref{equ.beta_PAM} in the literature. The main difficulty of solving equation \eqref{equ.beta_PAM} is the Anderson Hamiltonian $\fH = \Delta + \xi$. Due to the lack of regularity, it is not possible to define the product $\xi\varphi$ directly in dimension $2$ or $3$. The study of equation \eqref{equ.beta_PAM} then requires recently built  techniques such as regularity \cite{hairer2014theory} or the para-controlled distributions \cite{gubinelli2015paracontrolled}. In dimension 2, all of the solution theories require a renormalization, which means we need to remove a diverging constant in the Anderson Hamiltonian $\fH$ and consider actually the equation $\partial_t\varphi = (\fH-\infty)\varphi - \varphi^{1+\beta}$. It is understood as the limit of renormalized and discretized equation with 
\[
    \xi^n_e = \xi^n - c_n, c_n \simeq log(n).
\]
The second difficulty of solving equation \eqref{equ.beta_PAM} arises from the weight operation when solving the corresponding linear equation $\partial_t\varphi = \fH\varphi$ as in \cite{martin2019paracontrolled,hairer2015simple}. The solution theory of PAM uses heavily the logarithm growth of $\xi$ on the $\RR^2$. When we view \eqref{equ.beta_PAM} as 
\[
\partial_t\varphi = (\fH - \varphi^\beta)\varphi.
\]
The growth of $\xi - \varphi^\beta$ will be exponential and thus we can not apply directly the solution theory of PAM on $\RR^2$. However, we will see in section \ref{sec.variant} that the minus sign helps us build the existence and uniqueness results.

In recent work by Perkowski and Rosati \cite{perkowski2021rough}, they only obtain a partial existence and uniqueness result of equation \eqref{equ.beta_PAM} but this is enough for their application due to the $L^2$ bounded property from the finite variance off-spring distribution. However, their methods rely strongly on the $L^2$ martingale theory and can not be used directly for the infinite variance case. We, therefore, search for an alternative way, combining the idea of \cite{perkowski2021rough} and the classical approach via Laplace functional, to give the convergence in the infinite variance case. Finally, although we work in dimension 2, all results(with some change of parameters) are valid in dimension 1 with a much easier argument. 

\textbf{Structure of the work.} We only consider a deterministic environment with assumption \ref{ass.environment} since the construction of random environment from the deterministic environment is given in \cite{perkowski2021rough}. We introduce our model and the definition of $\beta-$rough super Brownian motion in section \ref{sec.model_formulation} and state our main results, i.e. the convergence results of the scaled branching random walk, a martingale problem formulation and the compact support property of the $\beta-$super Brownian motion. 

We study a variant of PAM ($\fH\rightarrow \fH - \phi$ for some $\phi$) and the non-linear equation like \ref{equ.beta_PAM} in section \ref{sec.variant} and show the existence and uniqueness of solution in the space with weight $e^{l|x|^\sigma}$ for any $l\in\RR,0\leq \sigma<1$. The result allows us to implement freely with the $\varphi$ in the Laplace functional $\EE[e^{-\langle\mu_t,\varphi\rangle}]$.

We then give the detailed proof of convergence of scaled branching random walk to the $\beta-$rough super Brownian motion by estimating the moments of $\sup_{0\leq s\leq t}\langle\mu_s,\varphi\rangle$. This estimation requires controlling branching random walks from a coupling method and passing the estimations to a simpler branching random walk. Then we apply the Aldous criterion to obtain the tightness.

Section \ref{sec.martingale_problem} is devoted to giving a martingale problem formulation of $\beta-$rough super Brownian motion. We finally then prove the compact support property that is proved in \cite{jin2023compact} for the finite variance case.

\section{Notation}
\subsection{Notation on regularity}
    We define the lattice $\ZZ_n^2 = \frac{1}{n}\ZZ^2$ and also denote $\ZZ_{\infty}^2 :=\RR^2$ for convenience. Similarly, we define $\TT_n^2:=(\RR/n\ZZ)^2$ and $\tilde{\TT}_n^2:=n\left(-\frac{1}{2},\frac{1}{2}\right]^2$ .We will consider weighted Besov space on $\ZZ_n^2$ introduced in \cite{martin2019paracontrolled}. 
    
    To motivate the definition, let's first introduce the definition of weight.
    \begin{definition}\label{def.weight}
   We denote by 
   \[
        \omega^{pol}(x) := \log(1+|x|),\qquad\omega_\sigma^{exp}(x) :=|x|^\sigma,
   \]
   where $x \in \RR^d,\sigma \in (0,1)$. For $\omega \in \pmb{\omega}:=\{\omega^{pol}\}\cup \{\omega_\sigma^{exp}|\sigma\in(0,1)\}$, we denote by $\pmb{\varrho}(\omega)$ the set of measurable, strictly positive $\rho:\RR^d\rightarrow(0,\infty)$such that
   \[
        \rho(x)\lesssim\rho(y)e^{\lambda\omega(x-y)},
   \]
   for some $\lambda = \lambda(\rho) > 0$. We also introduce the notation $\pmb{\varrho}(\pmb{\omega}):=\bigcup_{\omega\in\pmb{\omega}}\pmb{\varrho}(\omega)$. The objects $\rho \in \pmb{\varrho}(\pmb{\omega})$ will be called \emph{weights}.
\end{definition}
    The weight we considered in this paper is of the following form:
    \[
        p(a)(x) := (1+|x|)^a,\qquad e(l)(x):=e^{l|x|^\sigma},
    \]
    for non-negative $a,l$ and $0 < \sigma < 1$. We drop the index $\sigma$ since it does not influence the proof.
    
    For function $\varphi:\ZZ_n^2 \rightarrow \RR$, we define its Fourier transformation as function on torus $\TT_n^2$ by 
    \[
        \fF_n\varphi(k) := \frac{1}{n^2}\sum_{x\in \ZZ_n^2}\varphi(x)e^{-2\pi\iota k\cdot x},\qquad k\in\TT_n^2,
    \]
    and the inverse Fourier transformation for functions $\varphi: \TT_n^2 \rightarrow\RR$, 
    \[
        \fF_n^{-1}\varphi(x) := \int_{\TT_n^2} \varphi(k)e^{2\pi\iota k\cdot x}dk,\qquad x \in \ZZ_n^2.
    \]
    Let $\omega \in \pmb{\omega}$, it is defined in \cite{martin2019paracontrolled} the space $\fS_\omega$ consisting of all functions such that all derivatives of itself and its Fourier transformation decay like $e^{-\lambda \omega}$ for any $\lambda > 0$. The dual space of $\fS'_\omega$ is called ultra-distribution which will be the space the Besov spaces are built on.

Suppose $\rho_{-1},\rho_0 \in \fS'_\omega(\RR^2)$ are two non-negative and radial functions such that the support of $\rho_{-1}$ is contained in a ball $B \subset \RR^2$, the support of $\rho_0$ is contained in an annulus $\{x \in \RR^2: 0 < a \leq |x|\leq b\}$ and such that with 
\[
    \rho_{j} = \rho_{0}(2^{-j}\cdot),
\]
the following conditions are satisfied:
\begin{enumerate}
    \item $\sum_{i=-1}^\infty \rho_j(x) = 1,\forall x \in \RR^2$
    \item supp($\rho_i$) $\cap$ supp($\rho_j$) = $\emptyset$ if $|i-j| > 1$
\end{enumerate}
For each $\omega \in \pmb{\omega}$,  such a partition of unity always exists in $\fS'_\omega(\RR^2)$. We define furthermore an index $j_n$ to be the minimum index such that the support of $\phi_j$ intersects with $\tilde{\TT}_n^2$. Update the partition of unity of $\tilde{\TT}_n^2$ by defining periodic functions 
\[
    \rho^n_j = \rho_j, \qquad j < j_n, \qquad \rho^n_{j_n} = 1 - \sum_{j=-1}^{j_n-1}\rho_j.
\]
For a ultra-distribution $f\in\fS'_\omega(\ZZ_n^2)$, the Littlewood-Paley blocks of $f$ as 
\[
    \Delta^n_jf = \fF^{-1}(\rho^n_j\fF(f)).
\]
for $j = -1,0,\cdots,j_n$. 

For $\alpha\in \RR, p,q \in[1,\infty]$ and $\rho \in \pmb{\varrho}(\pmb{\omega})$, we define the weighted Besov space $B_{p,q}^\alpha(\ZZ_n^2,\rho)$ by
   \[
        B_{p,q}^\alpha(\ZZ_n^2,\rho):= \{f \in \fS'_\omega: \|f\|_{B_{p,q}^\alpha(\ZZ_n^2,\rho)} := \|(2^{j\alpha}\|\rho^{-1}\Delta_jf\|_{L^p})\|_{\ell^q} < \infty\}.
   \]
   In particular, the discrete H\"older space is defined by $C^\alpha(\ZZ_n^2,\rho):= B^{\alpha}_{\infty,\infty}(\ZZ_n^2,\rho)$ for $\alpha$ not an integer. The extension operator $\fE^n$ defined in \cite{martin2019paracontrolled} will be helpful when we consider the convergence from the discrete equation to the continuum equation.
    
    We also consider the space concerning the regularity with respect to time. Fix a time horizon $T>0$, we define the space $C_TC^\alpha(\ZZ_n^2,e(l))$ to be all continuous in time function $f:[0,T]\rightarrow C^\alpha(\ZZ_n^2,e(l+T))$ such that the norm
    \[
       \|f\|_{C_TC^\alpha(\ZZ_n^2,e(l))}:=\sup_{0\leq t\leq T}\|f(t)\|_{C^\alpha(\ZZ_n^2,e(l+t))},
    \]
    is finite. For $0 < \beta < 1$, the space $C^\beta_TC^\alpha(\ZZ_n^2,e(l))$ is defined by 
    \[
        \|f\|_{C^\beta_T C^\alpha(\ZZ^2_n,e(l))} := \|f\|_{C_TC^\alpha(\ZZ_n^2,e(l))} + \sup_{0\leq s< t\leq T} \frac{\|f(t)-f(s)\|_{C^\alpha(\ZZ_n^2,e(l+t))}}{t-s}.
    \]  
    We also define the space $\fM^\gamma C^\alpha(\ZZ_n^2,e(l))$ to be all function $f$ such that $t^\gamma f(t) \in C_TC^\alpha(\ZZ_n^2,e(l))$. And the space $\fL^{\gamma,\alpha}(\ZZ_n^2,e(l))$ consists of all functions $f$ such that 
    \[
        f \in \fM^\gamma C^\alpha(\ZZ_n^2,e(l)),\qquad t^\gamma f\in C^{\frac{\alpha}{2}}L^\infty(\RR^2,e(l)).
    \]
    We simply write $\fL^\alpha$ when $\gamma = 0$.
\subsection{Notation on operators}
    We give the definition of some discrete operators here. We define the discrete Laplace on lattice $\ZZ_n^2$ by 
    \[
        \Delta^n f(x) := n^2\sum_{y\sim x}(f(y) - f(x)).
    \]
    We will use $P_t^n$ to denote the discrete heat kernel, i.e. for any function $\varphi$, we denote $P_t^n\varphi$ to be the mild solution of the equation 
    \begin{equation}\label{equ.heat}
        \left\{
        \begin{array}{ll}
            \partial_t w^n_t = \Delta^n w^n_t,\\
            w^n_0 =\varphi.
        \end{array}
    \right.
\end{equation}
We also use $P_t$ for the continuum equation. Also, for a deterministic environment $\xi^n$(which will be introduced in assumption \ref{ass.environment}), we define the Anderson Hamiltonian $\fH^n = \Delta^n + \xi^n_e$, where $\xi^n_e = \xi^n - c_n$ with $c_n$ a renormalization constant to be introduced in the definition of deterministic environment. We then use the operator $T_t^n\varphi$ to indicate the solution to the equation 
    \begin{equation}\label{equ.PAM_homo}
        \left\{
        \begin{array}{ll}
            \partial_t w^n_t = \fH^nw^n_t,\\
            w^n_0 =\varphi.
        \end{array}
    \right.
\end{equation}
if the solution exists and is unique. All above definitions move to the continuous case with symbols $P_t,\Delta$ and $T_t,\fH$. There is another definition of $\fH^n$ which links to the continuous equation by the para-product. 
\begin{definition}
    For any two distributions $\varphi,\phi$ on $\ZZ_n^2$, the product can be decomposed as $\varphi\cdot\phi =\varphi \para \phi + \varphi\odot\phi + \phi\para\varphi$ such that 
    \[
        \varphi\para\phi = \sum_{j=-1}^{j_n} \Delta^n_{< j-1}\varphi\Delta^n_j\phi,\qquad \varphi\odot\phi = \sum_{\substack{|i-j|\leq 1\\-1 \leq i,j\leq j_n}}\Delta^n_i\varphi\Delta^n_j\phi,
    \]
    where $\Delta^n_{< i} = \sum_{j=-1}^i\Delta^n_j$. When $n=\infty$, it is exactly the same definition on $\RR^2$. This builds the relation between discrete and continuum cases.
\end{definition}
The definition of para-product allows us to define the para-controlled distribution. 
\begin{definition}
    Suppose $\kappa > 0$, $\xi^n$ and $\fI\xi^n$ satisfy the assumption \ref{ass.environment}. We say $\varphi^n$ is para-controlled if $\varphi^n \in C^{1-\kappa}(\ZZ_n^2,e(l))$ for some $l \in \RR$ and 
    \[
        \varphi^{n,\#}:= \varphi^n - \varphi^n \para \fI\xi^n \in C^{1+2\kappa}(\ZZ_n^2,e(l)).
    \]
\end{definition}
We define the para-controlled space $\fD(\ZZ_n^2, e(l))$ for all paracontrolled functions $\varphi^n$ such that
\[
    \|\varphi^n\|_{\fD(\ZZ_n^2,e(l))} = \|\varphi^n\|_{C^{1-\kappa}(\ZZ_n^2,e(l))} +  \|\varphi^{n,\#}\|_{C^{1+2\kappa}(\ZZ_n^2,e(l))} < \infty.
\]
We also define the domain $\fD_{\fH}$ of the operator $\fH$ on $\RR^2$, which is given by
\[
    \fD_{\fH}:=\left\{ \int_0^t T_su ds: t > 0, u \in C^\zeta(\RR^2,e(l)), l\in\RR,\zeta > 0\right\}.
\]  
The property of this set is discussed in \cite{perkowski2021rough}.
\subsection{General notations}
Throughout the paper, we will use $X^n$ to indicate the random walk on the lattice $\ZZ_n^2$ with the generator $\Delta^n$. 

We denote $\fM_F(\RR^2)$ as the space of finite measure on $\RR^2$ equipped with weak topology defined by family $C_b(\RR^2)$. We also let $\fM_F^{\pm}(\RR^2)$ be the signed finite measure on $\RR^2$.

For any Polish space $E$, we consider the space of $E-$valued càdlàg path $\DD([0,\infty),E)$ endowed with Skorohod topology.

\section{Model Formulation}\label{sec.model_formulation}
From now on, fix a parameter $1 < \beta < 2$. To simplify this paper, we only consider the deterministic environment $\xi^n$, which is simply a function from $\ZZ_n^2$ to $\RR$. We refer readers to \cite{perkowski2021rough} for how to deal with the situation when the environment $\xi^n$ is random. Let $\chi$ be a smooth function taking value 1 outside of $\left(-\frac{1}{4},\frac{1}{4}\right)^2$ and $0$ in $\left(-\frac{1}{8},\frac{1}{8}\right)^2$.
\begin{assumption}\label{ass.environment}
    For a deterministic environment $\xi^n$, define $\fI\xi^n$ to be the solution to the equation $-\Delta^n\fI\xi^n = \fF_n^{-1}(\chi\fF_n\xi^n)$. Consider a regularity parameter $\kappa' \in \left(0,\frac{1}{4}\right)$. We assume
    \begin{enumerate}
        \item There exists $\xi\in\cap_{a>0}C^{-1-\kappa'}(\RR^2,p(a))$ such that, for all $a > 0$,
        \[
            \sup_n\|\xi^n\|_{C^{-1-\kappa'}(\ZZ_n^2,p(a))} < \infty,\qquad \fE^n\xi^n\rightarrow\xi \text{ in }C^{-1-\kappa'}(\RR^2,p(a))
        \]
        \item The quantities $\frac{|\xi^n|}{n}$ is uniformly bounded over $n$. Furthermore, there exists $\nu \geq 0$ such that the following convergence holds
        \[
            \fE^n\frac{\xi_+^n}{n} \rightarrow \nu,\qquad \fE^n\frac{|\xi^n|}{n} \rightarrow 2\nu
        \]
        in $C^{-\delta}(\RR^2,p(a))$ for any $\delta >0, a > 0$.
        \item There exists $\{c_n\}\subset\RR$ with $\frac{c_n}{n}\rightarrow 0$ and there exists $\fI\xi\in\cap_{a>0} C^{1+\kappa'}(\RR^2,p(a))$ and $\fI\xi\diamond\xi \in \cap_{a>0}C^{-2\kappa'}(\RR^2,p(a))$ such that for all $a > 0$, 
        \[
            \sup_{n}\|\fI\xi^n\|_{C^{-1-\kappa'}(\ZZ_n^2,p(a))} + \sup_n\|(\fI\xi^n\odot\xi^n) - c_n\|_{C^{-2\kappa'}(\ZZ_n^2,p(a))} < \infty
        \]
        and $\fE^n\fI\xi^n \rightarrow \fI\xi$ in $C^{-1-\kappa'}(\RR^2,p(a))$ and $\fE^n(\fI\xi^n\odot\xi^n-c_n)\rightarrow \fI\xi\diamond\xi$ in $C^{-2\kappa'}(\RR^2,p(a))$.
    \end{enumerate}
\end{assumption}
We call a distribution $\xi$ on $\RR^2$ a deterministic environment if there exists $\xi^n$ satisfying the assumption \ref{ass.environment} and $\xi$ is the limit. Given a deterministic environment with assumption $\ref{ass.environment}$, we establish the particle systems approximating desired super-process.

 For each $n\in\NN$, we consider a branching random walk on the lattice $\ZZ_n^2$. In the language of branching random walk, the underlying motion is given by random walk $X^n$ with generator $\Delta^n$. The branching rate is given by 
\[
    a(x,dt) = |\xi^n(x)|dt,
\] 
for all $x \in \ZZ_n^2$. The branching mechanism is given by the generating function 
\begin{equation}
    g^n(x,s) = \frac{2\xi^n_+}{(1+\beta)|\xi^n|}(1-s)^{1+\beta} - \frac{\xi^n}{|\xi^n|} + \frac{2\xi_+^n}{|\xi^n|}s = \sum_{i=0}^\infty p_i\cdot q_i(x)s^i,
\end{equation}
where 
\[
    \sum_{i=0}^\infty p_i s^i = \frac{1}{1+\beta}(1-s)^{1+\beta} + s := g(s).
\]
One can verify that we have 
\[
    q_0(x) = \frac{(1-\beta)\xi^n_+ + (1+\beta)\xi^n_-}{|\xi^n|} \in \left(1-\beta,1+\beta\right),\qquad q_i(x) = \frac{2\xi^n_+}{|\xi^n|}\leq 2, i \geq 2.
\]
Since $\frac{\partial}{\partial s} g^n(x,0) = 0 = p_1$, we can set $q_1(x)$ to be any number. For example, set $q_1(x) = 1$. 

A detailed description can be given by defining the infinitesimal generator of a Markov process on the configuration space $E^n = (\NN)^{\ZZ_n^2}$ endowed with product topology. For any configuration $\eta \in E^n$, we define 
\[
    \eta^{x\rightarrow y}(z) = \1_{\{\eta(x) > 0\}}(\eta(z) + \1_{z=y} - \1_{z=x}),\qquad \eta^{x+k}(z) = 0\vee(\eta(z)+(k-1)\1_{z=x}).
\]
Furthermore, for any $F \in C_b(E^n)$, we define 
\[
    \Delta_x^n F(\eta) = n^2\left(\sum_{y\sim x}F(\eta^{x\rightarrow y}) - F(\eta)\right), \qquad d_x^{k} F(\eta) = F(\eta^{x+k}) - F(\eta).
\]
Since any configuration on $E^n$ can be viewed as a sigma finite point measure on $\RR^2$ and we eventually prove the tightness in the measure space, we will freely change from $E^n$ to the point measure on $\RR^2$ in the following content. 

To construct the approximation particle systems, we need furthermore the Poisson random measure introduced in appendix \ref{app.random}. 

\begin{definition}\label{def.discrete_mu}
    Fix a compact supported measure $\mu^n_0$ on $\ZZ_n^2$ and an average parameter $\varrho\leq \beta$. Let $\epsilon = n^{-\frac{1}{\varrho}}$. We define an $E^n-$valued stochastic process $Z^n(t)$ started at a Poisson random measure on $\ZZ_n^2$ with intensity $\frac{\mu^n_0}{\epsilon}$ with infinitesimal generator 
    \begin{equation}\label{equ.generator}
        \fL^n F(\eta) = \sum_{x\in\ZZ_n^2} \eta_x\left[\Delta_x^n F(\eta) + |a^n(x)|\sum_{k=0}^\infty p_kq_k(x)d_x^k F(\eta)\right],
    \end{equation}
    for any $F \in \fD(\fL^n)$ which consists of all $F\in C_b(E^n)$ such that the right-hand side of \eqref{equ.generator} is finite. Finally, set $\mu^n(t) = \epsilon Z^n(t)$ for all $t\in[0,\infty)$. 
\end{definition}
Our main results concern the limit behaviour of the measures $\mu^n$ when $\varrho = \beta$, which will be called $\beta-$rough super Brownian motion. Here is the definition of $\beta-$super Brownian motion. We use $\prescript{}{\varkappa}{U_t(\varphi)}$ to denote the unique solution to the equation 
    \begin{equation}\label{equ.nonlinear_PAM_RSBM}
            \left\{
        \begin{array}{ll}
            \partial_t w_t = \fH w_t - \varkappa w_t^{1+\beta},\\
            w_0 = \varphi.
        \end{array}
    \right.
    \end{equation}
In most situation, we will drop the prescript $\varkappa$ and write $U_t(\varphi)$ if there is no confusion.
\begin{definition}\label{def.beta_RSBM}
    Let $\xi$ be a deterministic environment satisfying assumption \ref{ass.environment}. Let $\varkappa > 0$ and $\mu_0$ be a compact support measure. Let $\mu$ be a process with values in the space $\DD([0,\infty),\fM_F(\RR^2))$, such that $\mu(0) = \mu_0$. Write $\fF = \{\fF_t\}_{t \in [0,\infty)}$ for the completed and right-continuous filtration generated by $\mu$. We call $\mu$ a $\beta-$rough super Brownian motion with parameter $\varkappa$ if it satisfies one of the three properties below:
    \begin{enumerate}
        \item For any $t\geq 0$ and non-negative function $\varphi \in C_c^\infty(\RR^2)$. The process 
        \[
            N_t^{\varphi}(s) := e^{-\langle\mu_s,U_{t-s}(\varphi)\rangle},
        \]
        is a bounded martingale.
        \item For any $t\geq 0$, $\varphi_0 \in C_c^\infty(\RR^2)$ and $f \in C_tC^\zeta(\RR^2,e(l))$ for any $\zeta > 0, l < -t$. Let $\varphi_t(s)$ solves 
        \[
            \partial_s\varphi_t(s) + \fH\varphi(s) = f(s), \qquad s\in[0,t],\varphi_t(t) = \varphi_0,
        \]
        then 
        \[
            \MM^{\varphi_0,f}_t(s):=\langle\mu_s,\varphi_t(s)\rangle - \langle\mu_0,\varphi_0\rangle - \int_0^s\langle\mu_r,f(r)\rangle dr,
        \]
        defined for $s\in[0,t]$, is a a $L^{1+\theta}$ bounded purely discontinuous martingale 
for any $0 < \theta < \beta$. In addition, the random point measure associated with $\MM^{\varphi_0,f}_t$ has the compensator
        \[
    \langle\mu_s,\varkappa|\varphi_t(s)|^{1+\beta}\rangle\frac{\beta(1+\beta)}{\Gamma(1-\beta)x^{\beta+2}}\1_{x\varphi_t>0}dsdx
\]
        \item For all function $\phi \in \fD_\fH$, 
        \[
    L^\phi_t := \langle{\mu_t,\phi}\rangle - \langle\mu_0,\phi\rangle - \int_0^t \langle\mu_s,\fH\phi\rangle ds,
\]
is a $L^{1+\theta}$ bounded purely discontinuous martingale for any $0 < \theta < \beta$. In addition, the random point measure associated with $L_t^{\phi}$ has the compensator
\[
    \langle\mu_t,\varkappa|\phi|^{1+\beta}\rangle\frac{\beta(1+\beta)}{\Gamma(1-\beta)x^{\beta+2}}\1_{x\phi>0}dtdx.
\]
    \end{enumerate}
\end{definition}
\begin{remark}
    It will be seen in the proof of \ref{thm.equivalent_def} that the second definition actually holds for $f$ with any weight $e(l)$ if we goes from 2. to 1. and obtain the moment estimation of $\langle\mu_t,e(l)\rangle$ from definition 1. via similar argument as discrete case. This moments estimation then ensure the limit procedure in the proof from 3. to 2. when the weight of $f$ is arbitrary.
\end{remark}
Our main results are the following:
\begin{theorem}\label{thm.equivalent_def}
    The three definition in \ref{def.beta_RSBM} are equivalent.
\end{theorem}
The proof is given in section \ref{sec.martingale_problem}
\begin{theorem}\label{thm.existence_RSBM}
    For any $\varkappa > 0$, the $\beta-$rough super Brownian motion exists and has unique law. Furthermore, it can be realized as a limit of some particles systems.
\end{theorem}
The proof is given in section \ref{sec.existence}. For the case $\varkappa < \frac{2}{1+\beta}$, it is the limit of $\mu_n$. And for more general case, we do the mixing method as in the paper \cite{perkowski2021rough}.

As a direct consequence of the method in \cite{jin2023compact} and a coupling method, we can show the compact support property and the super-exponentially persistence of the $\beta-$super Brownian motion.
\begin{definition}
    We say a stochastic process with value in $\fM_F(\RR^2)$ is super-exponentially persistent if, for any non-zero positive function $\varphi\in C_c^\infty(\RR^2)$ and for all $\lambda > 0$, it holds that 
    \[
        \PP\left[\lim_{t\rightarrow\infty}e^{-t\lambda}\langle\mu(t),\varphi\rangle = \infty\right] > 0.
    \]
\end{definition}

\begin{theorem}\label{thm.RSBM_SEP}
    The $\beta-$rough super Brownian motion is super-exponentially persistent.
\end{theorem}

\begin{definition}
    We say a stochastic process with value in $\fM_F(\RR^2)$ possesses the compact support property if 
    \[
        \PP\left[\bigcup_{0\leq s \leq t}supp(\mu_s)\text{ is compact}\right] = 1,
    \]
    for any $t \geq 0$.
\end{definition}

\begin{theorem}\label{thm.RSBM_CSP}
    The $\beta-$rough super Brownian motion possesses the compact support property.
\end{theorem}

\section{On variants of PAM}\label{sec.variant}
In this section, we discuss the required knowledge of SPDEs to construct the $\beta-$rough super Brownian motion. To start with, let's give the solution theory of PAM on $\ZZ_n^2$ and $\RR^2$. Details can be found in \cite{martin2019paracontrolled}. Let $\kappa' < \kappa\in \left(0,\frac{1}{4}\right)$ be a regularity parameter and $T > 0$ be the time horizon.
\begin{theorem}\label{thm.PAM}
    Consider $\varphi\in C^{1-\kappa}(\ZZ_n^2,e(l))$ and $\varphi^{\#}\in C^{1+2\kappa}(\ZZ_n^2,e(l))$. Then the equation  
    \begin{equation}\label{equ.PAM}
        \left\{
        \begin{array}{ll}
            \partial_t w^n_t = \fH^nw^n_t + f^n,\\
            w^n_0 =\varphi^n,
        \end{array}
    \right.
\end{equation}
admits a unique solution with the estimations
\begin{equation}\label{equ.PAM_estimation}
    \|w^n\|_{\fL_T^{1-\kappa}(\ZZ_n^2,e(l+T))} \lesssim \|f^n\|_{\fM^\gamma C^{-1+2\kappa}(\ZZ_n^2,e(l))} + \|\varphi\|_{\fD(\ZZ_n^2,e(l))}.
\end{equation}
The constant in $'\lesssim'$ is uniform over $n$. Furthermore, if the functions $\fE^n \varphi^n \rightarrow \varphi$ and $\fE^n f^n \rightarrow f$, then we have the convergence of solution $w^n$ to the solution of continuum PAM with the same estimation as \eqref{equ.PAM_estimation}.
\end{theorem}
\begin{remark}
    When the initial condition $\varphi$ is not paracontrolled but smooth, we could operate as follows. Consider $\tilde{w}_t = w_t - \varphi^n$, then the function $w_t$ satisfies 
    \begin{equation}\label{equ.PAM_smooth_initial}
                \left\{
        \begin{array}{ll}
            \partial_t \tilde{w}_t = \fH^n\tilde{w}_t + f^n + \fH^n\varphi^n\\
            \tilde{w}_0 = 0
        \end{array}
    \right.
    \end{equation}
    Now if $\varphi^n \in C^{1+2\kappa}(\ZZ_n^2,e(l))$ for some $l \in \RR$, then we see that $\fH^n\varphi^n$ is well-defined and belongs to the space $C^{-1+2\kappa}(\ZZ_n^2,e(l)p(a))$ for any $a > 0$. Then theorem \ref{thm.PAM} can be applied to equation \eqref{equ.PAM_smooth_initial} and obtain estimations on the solution $w_t$.
\end{remark} 
\subsection{Discrete case}
In this section, we fix a positive integer $n$ and consider equation on the discrete space-time $\RR_+\times\ZZ_n^2$. We now turn to the consideration of a non-linear discrete PAM.
    \begin{equation}\label{equ.nonlinear_PAM}
        \left\{
        \begin{array}{ll}
            \partial_t w^n_t = \fH^n w^n_t - f(w^n_t,\xi^n)w^n_t, & \text{ in }(0,\infty]\times\ZZ_n^2,\\
            w^n_0 =\varphi^n \geq 0, & \text{ on }\ZZ_n^2,\\
            w^n_t \geq 0, & \text{ }\forall t \geq 0,
        \end{array}
    \right.
\end{equation}
where $f$ is a non-negative function. We give assumptions on $f$ here:
\begin{assumption}\label{ass.nonlinearity}

    There exists some $\alpha > 0$, 
    \[
    |f(x,y)| \lesssim 1 + |x|^\alpha,
\] and the three-variables function 
\[
    \frac{f(x_1,y)x_1 - f(x_2,y)x_2}{x_1-x_2},
\]
is locally bounded on $(\ZZ_n^2)^3$.
\end{assumption}
For \eqref{equ.nonlinear_PAM}, we have the following theorem.
\begin{theorem}\label{thm.nonlinear_PAM}
    Suppose $f$ satisfies assumption \ref{ass.nonlinearity}. Let $\varphi^n \in C^{1+2\kappa}(\ZZ_n^2,e(l))$ for some $l \in \RR$. Then the equation \eqref{equ.nonlinear_PAM} admits a unique solution $w_t^n$ in the space $\cup_{\hat{l}\in\RR,\sigma\in(0,1)} \fL^{1-\kappa}(\ZZ_n^2,e(\hat{l}))$ and we have the estimation
    \[
        \|w^n_t\|_{\fL^{1-\kappa}(\ZZ_n^2,e(l'+(1+\alpha)t))} \lesssim \|\varphi^n\|_{C^{1+2\kappa}(\ZZ_n^2,e(l))}(1+\|\varphi^n\|_{C^{1+2\kappa}(\ZZ_n^2,e(l))}^\alpha),
    \]
    for any $l' >(1+\alpha)l$.
\end{theorem}

To prove theorem \ref{thm.nonlinear_PAM}, we first give a lemma on representing solutions of a larger class of equations by Dynkin's formula. 
\begin{lem}\label{lem.Dynkin_formula_PAM}
    Suppose $R:\RR^3\rightarrow\RR$ is continuous and has at most polynomial growth. Let $\phi \in C_TL^\infty(\ZZ_n^2,e(l))$ and $v_t\in C_TL^\infty(\ZZ_n^2, e(l))$ be the solution to the equation 
    \begin{equation}
    \left\{
        \begin{array}{ll}
            \partial_t v_t = (\Delta^n + \xi^n)v_t - R(v_t,\xi^n,\phi_t), &\text{ in }(0,\infty]\times\ZZ_n^2,\\
            v_0 = \varphi,&\text{ on }\ZZ_n^2,
        \end{array}
    \right.
    \end{equation}
    for some $l\in\RR$. Then we have for $x \in \ZZ_n^2$
    \begin{equation}
        v_t(x) = \EE\left[\varphi(X^n_t) + \int_0^t(v_{t-s}\cdot\xi^n)(X^n_s) - R(v_{t-s},\xi^n,\phi_{t-s})(X_s^n)ds|X_0 = x\right].
    \end{equation}
\end{lem}
\begin{proof}
    As said in \cite{martin2019paracontrolled}, all the space $C^\alpha(\ZZ_n^2,e(l))$ are equivalent when $n$ and $l$ are fixed. The assumption on $\xi^n$ then tells us that the weighted supremum norm of $\xi^n$ is bounded. Together with the assumption on $R$, we see that $\partial_tv_t \in C^\alpha(\ZZ_n^2,e(\hat{l}))$ for some $\hat{l}$ and any $\alpha\in\RR$. In particular $\nabla^n\partial_tv_t \in L^\infty(\ZZ_n^2,e(\hat{l}))$.
    
    For each $N > 0$ and fixed $t > 0$, consider truncated function $f^N(s,x) = v_{t-s}(x)\1_{|x| \leq N}$. We also define $f(s,x) = v_{t-s}(x)$. Since $X_s^n$ is a strong Markov process with generator 
    \[
        \Delta^n\varphi(x) = n^2\sum_{y\sim x}(\varphi(y)-\varphi(x)),
    \]
    from Dynkin's formula, we calculate
    \begin{equation}
        \begin{aligned}
            \EE\left[f(t,X_t^n) - f(0,x)\right] &=\sum_{i} \lim_{N\rightarrow\infty}\EE\left[f^N(t_{i+1},X_{t_{i+1}}^n) - f^N(t_i,X_{t_i}^n)\right] \\
            &=\sum_i \lim_{N\rightarrow\infty}\EE\left[\int_{t_i}^{t_{i+1}}\Delta^nf^N(t_{i+1},X_s^n) + \partial_tf^N(s,X_{t_i}^n)ds\right]\\
            &=\sum_i \EE\left[\int_{t_i}^{t_{i+1}}\Delta^nf(t_{i+1},X_s^n) + \partial_tf(s,X_{t_i}^n)ds\right].
        \end{aligned}
    \end{equation}
    The first and third equations hold because the bounds for $v,\partial_tv$ and $\forall l_0\in\RR$ 
    \[
        \EE\left[e^{l_0\sup_{0\leq s\leq t}|X_s^n|}\right] < \infty.
    \]
    The error terms can be estimated as 
    \begin{equation*}
        \begin{aligned}
            &\sum_i\EE\left[\left|\int_{t_i}^{t_{i+1}}\Delta^n\left(f(t_{i+1},X_s^n) - f(s,X_s^n)\right)ds\right|\right]\\
            &\lesssim
            \sum_i\EE\left[\int_{t_i}^{t_{i+1}}\int_s^{t_{i+1}}e^{l|X_s^n|^\sigma}duds\right] \lesssim_n \sum_i(t_{i+1} - t_i)^2,
        \end{aligned}
    \end{equation*}
    and, let $M_t^n:=\sup_{0\leq s\leq t}|X_s^n|$,
    \begin{equation}
        \begin{aligned}
            &\sum_i\EE\left[\left|\int_{t_i}^{t_{i+1}}\partial_tf^N(s,X_{s}^n) - \partial_tf^N(s,X_{t_i}^n)ds\right|\right] \\
            &\lesssim \sum_i \EE\left[\int_{t_i}^{t_{i+1}}|X_{t_i}^n -X_s^n|e^{\hat{l}|M_t^n|^\sigma}ds\right] \\
            &\lesssim \sum_i\int_{t_i}^{t_{i+1}}\EE[|X_{t_i}^n - X_s^n|^2]\EE[e^{2\hat{l}|M_t^n|^\sigma}]ds \lesssim \sum_i (t_{i+1} - t_i)^3.
        \end{aligned}
    \end{equation}
    Thus we obtain that 
    \[ 
        \EE\left[f(t,X_t^n) - f(0,x)\right] =  \EE\left[\int_{0}^{t}\Delta^nf(s,X_s^n) + \partial_tf(s,X_{s}^n)ds\right].
    \]
    Insert the equation of $v$, we then obtain the result.
\end{proof}

Next, we need a lemma on a variant equation of PAM, namely 
\begin{equation}\label{equ.variant_PAM}
            \left\{
        \begin{array}{ll}
            \partial_t w^n_t = (\fH^n - \phi^n_t)w^n_t, &\text{ in }(0,\infty]\times\ZZ_n^2,\\
            w^n_0 =\varphi^n\geq 0, &\text{ on }\ZZ_n^2,\\
            w^n_t \geq 0, & \text{ }\forall t \geq 0,
        \end{array}
    \right.
\end{equation}
where $\phi^n_t$ is a positive function.
\begin{lem}\label{lem.variant_PAM}
    Suppose non-negative functions $\phi^n_t \in C_TL^\infty(\ZZ^2_n,e(l_0+\cdot))$ and $\varphi^n \in C^{1+2\kappa}(\ZZ^2_n,e(l))$ for some $l_0,l \in \RR$. Then the equation \eqref{equ.variant_PAM} admits a minimum non-negative solution given by mild formulation 
    \[
        w_t^n = T_t^n(\varphi^n) + \int_0^t T^n_{t-s}(\phi^n_{s}w_s^n)ds,
    \]
    with estimation 
    \[
        \|w_t^n\|_{\fL^{1-\kappa}(\ZZ_n^2,e(l+l_0+2t))}\lesssim \|\varphi^n\|_{C^{1+2\kappa}(\ZZ_n^2,e(l))}(1+\|\phi_t^n\|_{L^\infty(\ZZ_n^2,e(l_0+t))}),
    \]
    and 
    \[
        \|w_t^n\|_{L^\infty(\ZZ_n^2,e(l+t))} \lesssim \|\varphi^n\|_{C^{1+2\kappa}(\ZZ_n^2,e(l))},
    \]
    where $'\lesssim'$ is independent of $n$. Furthermore, the uniqueness of solution holds in the space $\cup_{\hat{l}\in\RR} C_TL^\infty(\ZZ_n^2,e(\hat{l}))$.
\end{lem}
\begin{proof}
    By similar argument in \cite{gartner1990parabolic}, we know the minimum solution of equation \eqref{equ.variant_PAM} is given by the following Feynman-Kac formula 
    \begin{equation}
        w_t^n(x) = \EE\left[e^{\int_0^t(\xi^n_e - \phi_{t-s}^n)(X^n_s)ds}\varphi^n(X^n_t)|X^n_0 = x\right],
    \end{equation}
    if it is finite. By the positivity of $\phi^n$ and $\varphi^n$, we see the integrand has the estimation 
    \[
        0 \leq e^{\int_0^t(\xi^n_e - \phi_{t-s}^n)(X^n_s)ds}\varphi^n(X^n_t) \leq e^{\int_0^t\xi^n_e(X^n_s)ds}\varphi^n(X^n_t).
    \]
    Thus $w_t^n$ is finite due to the solvability of discrete PAM. Furthermore, we then have the estimation
    \[
        \|w_t^n\|_{L^\infty(\ZZ_n^2,e(l+t))} \lesssim  \|T_t^n\varphi^n\|_{L^\infty(\ZZ^2_n,e(l+t))} \lesssim \|\varphi^n\|_{C^{1+2\kappa}(\ZZ^2_n,e(l))}.
    \]
    By the regularity estimation of PAM, we have 
    \begin{equation*}
        \begin{aligned}
             \|w^n\|_{\fL^{1-\kappa}(\ZZ_n^2,e(l+l_0+2\cdot))}&\lesssim \|\varphi^n\|_{C^{1+2\kappa}(\ZZ_n^2,e(l+l_0))} + \|\phi^nw^n\|_{C_TL^\infty(\ZZ_n^2,e(l+l_0+2\cdot))} \\
             &\lesssim \|\varphi^n\|_{C^{1+2\kappa}(\ZZ_n^2,e(l))}(1+\|\phi_t^n\|_{L^\infty(\ZZ_n^2,e(l_0+t))}).
        \end{aligned}
    \end{equation*}
    Now from the Feynman Kac formulation to mild formulation, the integration by parts tells us that 
    \begin{equation*}
        \begin{aligned}
            w_t^n(x) &= \EE_{x}\left[\left( 1 - \int_0^t d_se^{\int_s^t(\xi_e^n-\phi_{t-r}^n)(X_r^n)dr}\right)\varphi^n(X_t^n)\right]\\
            &=\EE_{x}\left[\varphi^n(X^n_t) + \int_0^te^{\int_s^t(\xi_e^n-\phi^n_{t-u})(X^n_u)du}(\xi_e^n-\phi^n_{t-s})(X^n_s)ds \varphi^n(X^n_t)\right]\\
            &= \EE_{x}\left[\varphi^n(X^n_t) +\int_0^te^{\int_{t-s}^t(\xi_e^n-\phi^n_{t-u})(X^n_u)du}\varphi^n(X^n_t)(\xi_e^n-\phi^n_s)(X^n_{t-s})ds \right]\\
            &= \EE_{x}\left[\varphi^n(X^n_t) +\int_0^t w_s^n(X_{t-s}^n)(\xi_e^n-\phi^n_s)(X^n_{t-s})ds \right]\\
            &= P^n_t(\varphi^n) + \int_0^t P^n_{t-s}(w^n_{s}(\xi^n_e - \phi^n_{s}))ds,
        \end{aligned}
    \end{equation*}
    where we applied the Markov property of $X^n$ at the fourth equality. Then basic argument goes to the formulation with operator $T^n$. 
    
    On the uniqueness, if the solution $w^n \in C_TL^\infty(\ZZ_n^2,e(\hat{l}))$ for some $\hat{l}\in\RR$, we know by lemma \ref{lem.Dynkin_formula_PAM} that the solution must be given by 
    \begin{equation*}
        w^n_t(x) = \EE\left[\varphi(X^n_t) + \int_0^t(w_{t-s}\cdot\xi_e^n)(X^n_s) - w_{t-s}(X^n_r)(\phi^n_{t-s})(X_s^n)ds|X_0 = x\right],
    \end{equation*}
    which is exactly the mild formulation. Now we do the reverse procedure. The integration by parts formula gives us
    \begin{equation*}
        \begin{aligned}
            &\int_0^t e^{\int_{t-s}^t(\xi_e^n - \phi^n_u)(X_{t-u}^n)du}(\xi_e^n - \phi^n_{t-s})(X_s^n)\int_s^t w^n_{t-u}(\xi_e^n - \phi_{t-u}^n)(X_u^n)duds \\
            =& \int_0^t \int_s^t w_{t-u}^n(\xi_e^n - \phi_{t-u}^n)(X_u^n)du d\left(e^{\int_{t-s}^t(\xi_e^n - \phi^n_u)(X_{t-u}^n)du}\right)\\
            =& \int_0^t e^{\int_{t-s}^t(\xi_e^n - \phi^n_u)(X_{t-u}^n)du}w^n_{t-s}(\xi_e^n - \phi_{t-s}^n)(X_s^n)ds \\
            &- \int_0^tw_{t-u}^n(\xi_e^n-\phi^n_{t-u})(X_u^n)du.
        \end{aligned}
    \end{equation*}
    Thus we have by Markov property
    \begin{equation*}
        \begin{aligned}
            w_t^n(x) =& \EE_x\left[-\int_0^t e^{\int_{t-s}^t(\xi_e^n - \phi_u^n)(X_{t-u}^n)du}(\xi_e^n - \phi_{t-s}^n)(X_s^n)\left\{w_{t-s}^n(X_s^n) - \varphi^n(X_t^n) \right.\right. \\
            &- \left.\int_s^t w_{t-u}^n(\xi_e^n - \phi_{t-u}^n)(X_u^n)du\}ds + w_t^n(x)\right]\\
            =& \EE_x\left[w_t^n(x) - \int_0^tw_{t-s}^n(\xi_e^n-\phi_{t-s}^n)(X_s^n)ds \right.\\
            &+\left. \int_0^t\varphi^n(X_t^n)d\left(e^{\int_{t-s}^t(\xi_e^n - \phi_u^n)(X_{t-u}^n)du}\right)\right]\\
            =& \EE_x\left[\varphi^n(X_t^n)e^{\int_0^t(\xi_e^n - \phi^n_{t-s})(X_s)ds}\right].
        \end{aligned}
    \end{equation*}
    Thus solution $w_t^n$ is unique.
\end{proof}
\begin{remark}\label{rem.variant_PAM}
    Above lemma also holds when $\varphi^n \in \fD$ by theorem \ref{thm.PAM} and the estimation of $w_t^n$ will change to 
    \[
        \|w_t^n\|_{\fL^{1-\kappa}(\ZZ_n^2,e(l+l_0+2t))}\lesssim \|\varphi^n\|_{\fD(\ZZ_n^2,e(l))}(1+\|\phi_t^n\|_{L^\infty(\ZZ_n^2,e(l_0+t))}).
    \]
    Furthermore, we can see that the only solution to the equation \eqref{equ.variant_PAM} with initial $0$ should be $0$ by the proof of the above theorem even \textbf{without the positivity conditions (for $\phi^n, w^n$)}. 
\end{remark}
Now we are able to give a proof of theorem \ref{thm.nonlinear_PAM}
\begin{proof}[Proof of theorem \ref{thm.nonlinear_PAM}]
    We define the map $\fK: \phi^n\rightarrow \fK(\phi^n)$ to be the minimum solution to the equation \eqref{equ.variant_PAM}. Now consider a function $w \in C_T(\ZZ_n^2,e((l+\cdot)))$, we have 
    \[
    \|\fK(f(w,\xi))_t\|_{L^\infty(\ZZ_n^2,e(l+t))} \lesssim \|\varphi^n\|_{C^{1+2\kappa}(\ZZ_n^2,e(l))},
    \]
    This estimation shows that map $\fK(f(\cdot,\xi_e^n))$ maps the whole space $C_T(\ZZ_n^2,e((l+\cdot)))$ to a bounded subspace of itself. We now equip the space with the weak$*$ topology. From Schauder's fixed point theorem, we know there exists a fixed point for this map and the fixed point is exactly the mild solution to the equation \eqref{equ.nonlinear_PAM}. To show $w_t^n$ is indeed a solution of equation \eqref{equ.nonlinear_PAM}. We could write $w_t^n$ in the Feynman Kac formulation 
    \[
        w_t^n(x) = \EE_x\left[e^{\int_0^t(\xi_e^n - f(w_{t-s}^n,\xi^n_e))(X_s)ds}\varphi^n(X_t)\right],
    \]
    then by the lemma \ref{lem.variant_PAM}, we know $w_t^n$ is the minimum solution to the equation 
\begin{equation}
            \left\{
        \begin{array}{ll}
            \partial_t u^n_t = (\fH^n - f(w_t^n,\xi^n_e))u^n_t,\\
            u^n_0 =\varphi^n,\\
            u^n_t \geq 0.
        \end{array}
    \right.
\end{equation}
Thus is it a solution to the equation \eqref{equ.nonlinear_PAM}. Now we turn to the uniqueness property. Suppose we have two solutions to the equation \eqref{equ.nonlinear_PAM}, say $\tilde{w}^n$ and $w^n$. Then the differences 
\[
    \bar{w}^n := \tilde{w}^n - w^n,
\]
satisfies the equation
\begin{equation}
            \left\{
        \begin{array}{ll}
            \partial_t \bar{w}^n_t = \fH^n\bar{w}^n_t - \frac{f(\tilde{w}_t^n,\xi^n_e)w^n_t - f(w_t^n,\xi^n_e)w^n_t}{\bar{w}_t^n}\bar{w}^n_t,\\
            \bar{w}^n_0 =0.
        \end{array}
    \right.
\end{equation}
Above estimation shows that supremum norm of $\tilde{w}^n$ is bounded, and combining with assumption \ref{ass.nonlinearity}, we know 
\[
     \frac{f(\tilde{w}_t^n,\xi^n_e)w^n_t - f(w_t^n,\xi^n_e)w^n_t}{\bar{w}_t^n} \in C_TL^\infty(Z_n^2,e(\tilde{l}).
\]
Then from remark \ref{rem.variant_PAM}, we know $\bar{w}^n = 0$ and hence the uniqueness holds.
\end{proof}

\subsection{Continuous case}
For the continuous case, all the existence problems of mild solutions can be solved by sending $n$ to infinity in the discrete since all the estimations in the discrete cases are uniform over $n$. We are then interested in the uniqueness of solutions. 

As we can see in the discrete case, the keys to the uniqueness are the equivalence between Feynman-Kac formula and the mild formulation. It is not clear what $e^{\int_0^t\xi(X_s)ds}$ means in Feynman-Kac formula for the continuous case. However, we can hide this quantity by considering the directed polymer measure directly.  

The construction and convergence of polymer measure for spatial white noise by mollified or discrete white noise on $\TT^2$ are given in \cite{cannizzaro2018multidimensional,chouk2017invariance}. We can see the argument in \cite{chouk2017invariance} is closely related to the convergence of solutions to the corresponding Parabolic Anderson Models and the application of \cite{martin2019paracontrolled} on the whole $\RR^2$ will then results in the construction and convergence of polymer measure on $\RR^2$.  

Intuitively, we only need the Markov property of $X^n$ when we prove the uniqueness of equation \eqref{equ.variant_PAM}. And in \cite{alberts2014continuum}, the authors show the polymer measure conditioned on the environment is actually a Markov process, which implies that the uniqueness in the continuous case should also hold.  We start with the definition of polymer measure $Q$ on the space $C[0,T]$. 
\begin{definition}
    Let $X_t$ be the canonical process of $C[0,T]$, define the finite dimensional distribution for $0 = t_0 < t_1 < \cdots < t_k \leq T$ and $x_0 = x,x_i\in\RR^2,i=1,\cdots,k$:
    \[
        Q_x\left[X_{t_i} \in dx_i\right] = \frac{1}{\fZ_T(x)}\prod_{i=0}^{k-1}\fZ_{t_{i+1}-t_i}(x_i,x_{i+1})\fZ_{T-t_k}(x_k)dx_1dx_2\cdot dx_k,
    \]
    where $\fZ_t(x,y)$ is the solution to the equation \begin{equation}\label{equ.PAM_dirac}
            \left\{
        \begin{array}{ll}
            \partial_t \fZ_t(x,y) = \fH_x\fZ_t(x,y), & \text{ in }(0,\infty]\times\RR^2,\\
            \fZ_0(x) =\delta_y(x), &\text{ on }\RR^2.
        \end{array}
    \right.
\end{equation}
and $\fZ_t(x) := \int_{\RR^2}\fZ_t(x,y)dy$. The measure $Q_x$ generated by the above finite-dimensional distribution is called the directed polymer measure starting at point $x$.
\end{definition}
\begin{remark}
    The two-dimensional dirac function have the regularity of $B_{p,\infty}^{-2(1-1/p)}$ and applying conclusion in \cite{martin2019paracontrolled} with $p=1$, we see that equation \eqref{equ.PAM_dirac} admits a unique solution in the paracontrolled space. The solution $\fZ_t(x,y)$ is actually a Green's function and we have the representation 
    \[
        T_t\varphi(x) = \int_{\RR^2}\fZ_t(x,y)\varphi(y)dy.
    \]
    Furthermore, the consistency of finite-dimensional distribution is guaranteed by 
    \[
        \fZ_{t+s}(x,y) = (T_{t+s}\delta_y)(x) = (T_t\fZ_s(\cdot,y))(x) = \int_{\RR^2}\fZ_t(x,z)\fZ_s(z,y)dz.
    \]
    In addition, the markov property of polymer measure can be seen from its finite-dimensional distribution directly.
\end{remark}
Apart from the Markov property, we need a result on the transition function of the polymer measure, which can actually be viewed as a version of definition of polymer measure.
\begin{lem}
    We have the following representation for $\varphi \in C^{1+2\kappa}(\RR^2,e(l))$:
    \[
        T_{t-s}\varphi(X_s) = \EE_x^{Q}\left[\fZ_{T-s}(X_s)\frac{\varphi(X_t)}{\fZ_{T-t}(X_t)}|X_s\right],
    \]
    for $X_s$ following the probability measure $Q_x$.
\end{lem}
\begin{proof}
    For any measurable function $g$, we calculate 
    \begin{equation*}
        \begin{aligned}
            &\EE_x^Q\left[\fZ_{T-s}(X_s)\frac{\varphi(X_t)}{\fZ_{T-t}(X_t)}g(X_s)\right] \\
            =& \frac{1}{\fZ_T(x)}\int_{\RR^2}\fZ_s(x,x_1)\fZ_{t-s}(x_1,x_2)\fZ_{T-t}(x_2)\fZ_{T-s}(x_1)g(x_1)\frac{\varphi(x_2)}{\fZ_{T-t}(x_2)}dx_1dx_2\\
            =& \frac{1}{\fZ_T(x)}\int_{\RR^2}\fZ_s(x,x_1)\fZ_{t-s}(x_1,x_2)\fZ_{T-s}(x_1)g(x_1)\varphi(x_2)dx_1dx_2\\
            =&\frac{1}{\fZ_T(x)}\int_{\RR}T_{t-s}\varphi(x_1)g(x_1)\fZ_s(x,x_1)\fZ_{T-s}(x_1)dx_1 = \EE^Q_x[T_{t-s}\varphi(X_s)g(X_s)].
        \end{aligned}
    \end{equation*}
    Thus the result follows.
\end{proof}

Now consider equation
\begin{equation}\label{equ.variant_PAM_continuum}
            \left\{
        \begin{array}{ll}
            \partial_t w_t = (\fH - \phi_t)w_t, & \text{ in }(0,\infty]\times\RR^2,\\
            w_0 =\varphi\geq 0, &\text{ on }\RR^2\\
            w_t \geq 0, &\text{ }\forall t \geq 0.
        \end{array}
    \right.
\end{equation}
\begin{lem}
    Suppose $\phi_t \in C_TL^\infty(\RR^2,e(l_0+\cdot))$ and $\varphi \in C^{1+2\kappa}(\RR^2,e(l))$ for some $l_0,l \in \RR$. Then the equation \eqref{equ.variant_PAM} admits a solution given by mild formulation 
    \[
        w_t = T_t(\varphi) + \int_0^t T_{t-s}(\phi_{s}w_s)ds,
    \]
    with estimation 
    \[
        \|w_t\|_{\fL^{1-\kappa}(\RR^2,e(l+l_0+2t))}\lesssim \|\varphi\|_{C^{1+2\kappa}(\RR^2,e(l))}(1+\|\phi_t\|_{L^\infty(\RR^2,e(l_0+t))}),
    \]
    and 
    \[
        \|w_t\|_{L^\infty(\RR^2,e(l+t))} \lesssim \|\varphi\|_{C^{1+2\kappa}(\RR^2,e(l))}.
    \]
    Same result for $\varphi\in\fD^{1-\epsilon}(\RR^2,e(l))$. Furthermore, the uniqueness of solution holds in the space $\cup_{\hat{l}\in\RR,\sigma\in(0,1)} C_TC^0(\RR^2,e(\hat{l}))$.
\end{lem}
\begin{proof}
    The existence is a direct consequence of discrete case and taking the limit. On the uniqueness, the mild formulation is given by 
    \begin{equation*}
        \begin{aligned}
            w_t(x) =& T_t\varphi - \int_0^t T_{t-s}(\phi_sw_s)ds \\
            =& \EE^Q_x\left[\fZ_T(x)\frac{\varphi(X_t)}{\fZ_{T-t}(X_t)} - \fZ_T(x)\int_0^t \frac{\phi_{t-s}w_{t-s}(X_s)}{\fZ_{T-s}(X_s)}ds\right].
        \end{aligned}
    \end{equation*}
    Furthermore, let $0 < s < t$, we have 
    \begin{equation*}
        \begin{aligned}
        w_{t-s}(X_s) = \EE_x^Q\left[\fZ_{T-s}(X_s)\frac{\varphi(X_t)}{\fZ_{T-t}(X_t)} - \fZ_{T-s}(X_s)\int_s^t\frac{\phi_{t-u}w_{t-u}(X_u)}{\fZ_{T-u}(X_u)}du\large|X_s\right].
        \end{aligned}
    \end{equation*}
    Thus, same calculation as in the lemma \ref{lem.variant_PAM} gives 
    \begin{equation*}
        \begin{aligned}
            w_t(x) =& \EE_x^Q\left[w_t(x) + \fZ_T(x)\int_0^t e^{\int_{t-s}^t-\phi_u(X_{t-u})du}\frac{-\phi_{t-s}(X_s)}{\fZ_{T-s}(X_s)}\left\{w_{t-s}(X_s) - \right.\right.\\
            & \left.\left.\fZ_{T-s}(X_s)\frac{\varphi(X_t)}{\fZ_{T-t}(X_t)} + \fZ_{T-s}(X_s)\int_s^t\frac{\phi_{t-u}(X_u)w_{t-u}(X_u)}{\fZ_{T-u}(X_u)}du\right\}ds\right]\\
            =& \EE_x^Q\left[\fZ_T(x)\frac{\varphi(X_t)}{\fZ_{T-t}(X_t)}\left(1 - \int_0^te^{-\int_{t-s}^t\phi_{u}(X_{t-u})dy}\phi_{t-s}(X_s)ds\right)\right]\\
            =&\EE_x^Q\left[\fZ_T(x)e^{-\int_0^t\phi_{t-u}(X_u)du}\frac{\varphi(X_t)}{\fZ_{T-t}(X_t)}\right].
        \end{aligned}
    \end{equation*}
    The continuity of $w_t, X_t$ and the strict positivity of $\fZ$ ensure all the above calculations inside the expectation are valid. Thus we obtain the uniqueness of mild solution to the equation \eqref{equ.variant_PAM_continuum}. 
\end{proof}
It is then a direct consequence that the solution of equation 
\begin{equation}\label{equ.nonlinear_PAM_continuum}
            \left\{
        \begin{array}{ll}
            \partial_t w_t = \fH w_t - f(w_t)w_t,  &\text{ in }(0,\infty]\times\RR^2,\\
            w_0 =\varphi\geq 0, &\text{ on }\RR^2,\\
            w_t \geq 0,&\text{ }\forall t\geq 0,
        \end{array}
    \right.
\end{equation}
is unique, where $f$ is a non-negative function.
\begin{cor}
    Suppose $|f(x)|\leq |x|^\alpha$ for some $\alpha > 0$, $xf(x)$ is locally Lipschitz and $\varphi\in C^{1+2\kappa}(\RR^2,e(l))$, then there exists a unique mild solution to the equation \eqref{equ.nonlinear_PAM_continuum} and we have the estimation
    \[
        \|w_t\|_{\fL^{1-\kappa}(\RR^2,e(l'+(1+\alpha)t))} \lesssim \|\varphi\|_{C^{1+2\kappa}(\RR^2,e(l))}(1+\|\varphi\|_{C^{1+2\kappa}(\RR^2,e(l))}^\alpha),
    \]
    for any and $l' >(1+\alpha)l$.
\end{cor}
\begin{proof}
    The existence comes from the approximation by discrete equation. Uniqueness comes from the similar argument as in the proof of theorem \ref{thm.nonlinear_PAM}.
\end{proof}

\begin{remark}
    All the results apply to the inhomogeneous equations:
    \begin{equation}\label{equ.variant_PAM_continuum_inhomo}
            \left\{
        \begin{array}{ll}
            \partial_t w_t = (\fH - \phi_t)w_t + g_t, &\text{ in }(0,\infty]\times\RR^2,\\
            w_0 =\varphi\geq 0, &\text{ on }\RR^2,\\
            w_t \geq 0&\text{ }\forall t\geq 0,
        \end{array}
    \right.
\end{equation}
and 
\begin{equation}\label{equ.nonlinear_PAM_continuum_inhomo}
            \left\{
        \begin{array}{ll}
            \partial_t w_t = \fH w_t - f(w_t)w_t + g_t, &\text{ in }(0,\infty]\times\RR^2,\\
            w_0 =\varphi\geq 0, &\text{ on }\RR^2,\\
            w_t \geq 0&\text{ }\forall t\geq 0,
        \end{array}
    \right.
\end{equation}
with the additional condition $g \in C_TL^\infty(\RR^2,e(l))$ and non-negative. Since in these case, the mild solution of \eqref{equ.variant_PAM_continuum_inhomo} can be expressed as 
\begin{equation}
    \begin{aligned}
        w_t(x) = \EE_x^Q&\left[\fZ_T(x)e^{-\int_0^t\phi_{t-u}(X_u)du}\frac{\varphi(X_t)}{\fZ_{T-t}(X_t)}\right. \\
        &\quad+ \left.\fZ_T(s)\int_0^te^{-\int_0^s\phi_{t-u}(X_u)du}\frac{g_{t-s}(X_s)}{\fZ_{T-s}(X_s)}ds\right].
    \end{aligned}
\end{equation}
Thus the comparison principle with respect to $\phi$ still holds. The result is needed for the compact support property.
\end{remark}

\begin{remark}
    All the above discussion is valid for the operator $\fH$ replaced by $\Delta$ with much simpler argument and better regularity. 
\end{remark}

\section{Existence as limit of particle systems}\label{sec.existence}
\subsection{For the case $\varkappa < \frac{2}{1+\beta}$}
In this section, we will show that the family $\{\mu^n\}$ defined in definition \ref{def.discrete_mu} is tight in $\DD([0,\infty),\fM(\RR^2))$ and its limit is the $\beta-$ rough super Brownian motion with desired Laplace functional.

We will need auxiliary branching random walks given by the same branching rate $a^n$ with branching mechanism $g(s)$, which we denote the empirical measure as $\{\tilde{\mu}^n\}$.

Notice this branching mechanism $g$ is independent with respect to the $n$ and position $x$ and is well studied in classical books\cite{dawson1992infinitely,dawson1993measure} when the branching rate is also independent of the position. 

\begin{theorem}\label{thm.existence_laplace}
    Suppose $\mu_0$ is supported on a compact set. The limit $\mu_t:= \lim_{n\rightarrow\infty}\mu_t^n$ exists as stochastic processes in the space $\DD([0,\infty),\fM(\RR^2))$ and, when $\varrho = \beta$, for all non-negative function $\varphi \in C^{1+2\kappa}(\RR^2,e(l)),l\in\RR$, the laplace transformation is given by 
    \begin{equation}
        \EE\left[e^{-\langle\mu_t,\varphi\rangle}\right] = e^{-\langle\mu_0,U_t(\varphi)\rangle},
    \end{equation}
    where $U_t(\varphi)$ is the unique non-negative solution to the equation \eqref{equ.nonlinear_PAM_RSBM} with $\varkappa = \frac{2\nu}{1+\beta}$. In particular, this shows $\mu_t$ is markov and the first definition of definition \ref{def.beta_RSBM} is satisfied. Thus the $\beta-$rough super Brownian motion exists and is unique in law. In addition, when $\varrho < \beta$, the limit $\mu_t$ satisfies the equation PAM.
\end{theorem}

Let's first give a more detailed description of spatial dependent branching mechanism and branching rate here. Since all particles in the system move independently, it is sufficient to describe the motion of one particle. To describe the branching time of one particle, let's consider independent Poisson processes $N_t^x$ with intensity $a(x)$ attached at each site $x$ of $\ZZ_n^2$. Let $X_t^n$ be a simple random walk on $\ZZ_n^2$, then the branching time $\tau$ can be defined by 
\[
    \tau := \inf\{t \geq 0: X_t^n \text{ triggers the clock of }N_t^{X_t^n} \text{ at time }t\}.
\]
By the strong Markov property of Poisson process, we know that each time $X^n$ changes site, say time $t_0$ , there is a new Poisson process $(N_{t}^x - N_{t_0}^x)_{t\geq t_0}$ which is independent of all past event determining the branching time. 
Now let's do a scaling: set $\tilde{N}_t^x = N_{\frac{t}{a(x)}}^x$, then $\tilde{N}_t^x$ are Poisson processes with intensity $1$. If we paste all the pieces of scaled Poisson processes before $X_t^n$ branches, we obtain a random variable following the law of exponential distribution with parameter 1. Hence, by the scaling time, we can give a new definition of branching time $\tau$: 
\[
    \tau:= \inf\left\{t \geq 0: \int_0^t a(X_s)ds \geq \tilde{\tau}\right\},
\]
where $\tilde{\tau}$ is of exponential distribution with parameter $1$. This also leads to the law of branching time $\tau$, which is given by 
\begin{equation}\label{equ.branching_time}
    \PP[\tau \geq t] = e^{-\int_0^ta(X_s,ds)}.
\end{equation}

As the consequence of Poisson Cluster Random Measure \ref{lem.Poisson_cluster}, we are able to give the Laplace functional of $\mu_t^n$ and $\tilde{\mu}_t^n$.
\begin{lem}
    Let $A = \fH^n$ or $\Delta^n, B = \frac{2\xi_+^n\epsilon^{\beta}}{1+\beta}$or $\frac{|\xi^n|\epsilon^{\beta}}{1+\beta}$ correspondingly. For each $\varphi$, there exists $U_t^n(\varphi)$ or $\tilde{U}_t^n(\varphi)$ being the unique solution to the equation 
    \begin{equation}\label{equ.approx_laplace_PDE}
        \left\{
        \begin{array}{ll}
            \partial_t v^n_t = A v_t - B(v_t^n)^{1+\beta}, & \text{ in }(0,\infty]\times\RR^2,\\
            v^n_0 = \frac{1-e^{-\epsilon\varphi(x)}}{\epsilon}, & \text{ on }\RR^2.
        \end{array}
    \right.
\end{equation}
such that the Laplace functional of $\mu_t^n$ and $\tilde{\mu_t}^n$ is given by 
\[
    \EE[e^{-\langle \mu_t^n,\varphi\rangle}] = e^{-\langle \mu^n_0,U_t^n(\varphi)\rangle},\qquad \EE[e^{-\langle \tilde{\mu}_t^n,\varphi\rangle}] = e^{-\langle \mu^n_0,\tilde{U}_t^n(\varphi)\rangle}.
\]
\end{lem}
\begin{proof}
    We only show the proof for $\mu_t^n$, the proof to $\tilde{\mu}_t^n$ is similar.
    
    Define 
\begin{equation}
    w_{t}^n(x):= \EE\left[e^{-\langle\mu_t^n,\varphi\rangle}|Z_0^n = \delta_x\right].
\end{equation}
    Then lemma \ref{lem.laplace_equi} and lemma \ref{lem.laplace_equation_formula} tell us that 
\begin{equation}
    w_{t}^n(x):= \EE\left[e^{-\epsilon\varphi(x)} + \int_0^t|\xi^n|(X^n_r)g^n(X^n_r,w^n_{t-r}(X_r))dr - \int_0^t|\xi^n|w_{t-r}^n(X_r)dr\right].
\end{equation}
    The Poisson cluster formula \ref{lem.Poisson_cluster} then gives us that 
\begin{equation}
    \EE[e^{-\langle \mu_t^n,\varphi\rangle}] = e^{-\left\langle\mu^n_0,\frac{1-w^\epsilon_{t}}{\epsilon}\right\rangle}.
\end{equation}
Now define 
\[
    v^n_{t}(x) = \frac{1 - w_{t}^n(x)}{\epsilon}.
\]
The calculation shows that 
\begin{equation}\label{equ.laplace_discrete}
    v_{t}^n(x) = \EE\left[\frac{1-e^{-\epsilon\varphi(X^n_t)}}{\epsilon} + \int_0^t v^n_{t-r}\xi(Z_r) - \frac{1}{1+\beta}(v^n_{t-r})^{1+\beta}(2\xi_+^n\epsilon^{\beta})ds\right].
\end{equation}
Thus, $v_{t}^n$ is the unique solution to the equation \eqref{equ.approx_laplace_PDE} by lemma \ref{lem.variant_PAM}. 
\end{proof}

Next goal is to prove the tightness of $\mu^n$ in the space $\DD(\RR^+,\fM(\RR^2))$. We need following lemmas.
\begin{lem}\label{lem.coupling}
    For any $0 < \beta < 1$, there exists an integer $N(\beta)$ such that there exists a coupling of process $\mu_t^n$ and $N(\beta)$ sum of independent $\tilde{\mu}_t^{n,i}$ being copies of $\tilde{\mu}_t^n$ so that we have 
    \[
        \langle\mu_t^n,\varphi\rangle \leq \sum_{i=1}^{N(\beta)}\langle\tilde{\mu}_t^{n,i},\varphi\rangle,
    \]
    for any non-negative $\varphi$. Furthermore, $\langle\tilde{\mu}_t^{n,i},1\rangle$ are martingales.
\end{lem}
\begin{proof}
    Consider random variables $Z^n_0(x)$ and $Z^n_i$ whose generators are given by $g^n(x,s)$ and $g(s)$. We show that there exists $N(\beta)$ such that for any $k \in \NN$, the inequality holds 
    \begin{equation}\label{equ.controll}
        \PP[Z_0^n(x)\geq k] \leq \PP\left[\sum_{i=1}^{N(\beta)}Z^n_i \geq k\right].
    \end{equation}
    By the definition of generator $g^n$, we rewrite the generator into infinite series 
    \[
        g^n(x,s) = \sum_{k=0}^\infty p_kg_k(x)s^k,
    \]
    and 
    \[
        g(s) = \sum_{k=0}^\infty p_ks^k.
    \]
    We know furthermore that $g_0 \geq \frac{1-\beta}{1+\beta}$ and $g_k\leq 2$ for $k\geq 1$. Now consider $Z_1^n + Z_2^n + Z_3^n$, we have for all $k \in \NN$,
    \[
        \PP[Z_1^n + Z_2^n + Z_3^n\geq k] \geq 3\PP[Z_1^n\geq k]\PP[Z_1^n < k]^2.
    \]
    Since $\lim_{k\rightarrow\infty} \PP[Z_1^n < k] = 1$, we know there exists $k_0$ such that 
    \[
        \PP[Z_1^n + Z_2^n + Z_3^n\geq k_0] \geq 2\PP[Z_1^n\geq k_0] \geq \PP[Z_0^n(x) \geq k_0].
    \]
    Now for the $k < k_0$, due to $p_0g_0(x) > 0$, the central limit theorem tells us the existence of $N(\beta)$ that \eqref{equ.controll} holds. Then the result follows since we can always produce more particles at each site using inequality \eqref{equ.controll}.
    
    The martingale property of $\langle\tilde{\mu}_t^{n,i},1\rangle$ follows from the fact that $g'(1) = 1$, which means that each time we branch, the expectation of number of off-spring is again $1$. So the total number of the particle system should form a martingale. 
\end{proof}

\begin{lem}\label{lem.martingale}
    Let $l$ be any real number. For $\varphi\in C^{1+2\gamma}(\ZZ_n^2,e(l))$, the process 
\begin{equation}
    M_t^{n,\varphi}(s):=\mu_s^n(T_{t-s}^n\varphi) - \epsilon Z_0^n(T_t^n\varphi),\qquad  \tilde{M}_t^{n,\varphi}(s):= \tilde{\mu}_s^n(P_{t-s}^n\varphi) - \epsilon Z_0^n(P_t^n\varphi),
\end{equation}
are martingales. Thus by the formula of Poisson random measure, 
\[
    \EE[\langle \mu_t^n,\varphi\rangle] = \langle\mu_0^n,T^n_t\varphi\rangle,\qquad \EE[\langle\tilde{\mu}_t^n,\varphi\rangle] = \langle\mu_0^n,P_t^n\varphi\rangle.
\]
\end{lem}
\begin{proof}
    We consider only $\mu_t^n$ here, the argument for $\tilde{\mu}_t^n$ is similar. For any time dependent function $\phi_\cdot\in C^1_tL^\infty(\ZZ_n^2,e(-t+\cdot))$. Consider the function $F_\phi^t: [0,t]\times (\NN^{\ZZ_n^2})\rightarrow \RR$ defined as 
    \[
        F_\phi^t(s,\eta) = \sum_{x\in\ZZ_n^2} \phi_s(x)\eta_x.
    \]
    Apply Dynkin's formula to $F_\phi^t$ and do similar things as in the proof of lemma \ref{lem.Dynkin_formula_PAM}, we obtain that 
    \[
        M_t(s):=\mu_s^n(\phi_s) - \epsilon Z_0^n(\phi_0) - \int_0^s \partial_r F_\phi^t(r,\mu_r^n) + \fL^nF_\phi^t(r,\mu_r^n)dr,
    \]
    is a martingale. The calculation gives 
    \[
        \partial_rF_\phi^t(r,\mu_r^n) = \mu_r^n(\partial_r\phi_r),\qquad  \fL^nF_\phi^t(r,\mu_r^n) = \mu_r^n(\fH^n\phi_r).
    \]
    Consider $\phi_s = T^n_{t-s}\varphi$, we have the equality 
    \[
        \partial_r\phi_r + \fH^n\phi_r = 0.
    \]
    When $\varphi$ is compactly supported, the result follows since now we have $\phi_\cdot \in C_t^1L^\infty(\ZZ_n^2,e(-t+\cdot))$. For more general $\varphi \in C^{1+2\kappa}(\ZZ_n^2,e(l))$, we only need to consider when $\varphi$ is non-negative since the general $\varphi$ can be controlled by some non-negative function in the same space. 
    
    We can then approximate $\varphi$ by an increasing sequence of compactly supported functions $\varphi_m$ such that $\varphi_m \rightarrow \varphi \in C^{1+2\kappa}(\ZZ_n^2,e(l+1))$. This is visible since all regularities are the same for the discrete spaces and increasing weight a little bit can make $\varphi$ to be $0$ at the infinity under this weight. Then the solution theory of PAM or heat equation gives us that 
    \[
        T_t\varphi_m \rightarrow T_t\varphi \in C^{1-\kappa}(\ZZ_n^2,e(l+1)).
    \]  
    Then by monotone convergence theorem, we have 
    \[
        \EE[\langle\mu_t^n,\varphi\rangle] = \lim_{m\rightarrow\infty}\EE[\langle\mu_t^n,\varphi_m\rangle] = \lim_{m\rightarrow\infty}\langle\mu_0^n,T_t^n\varphi_m\rangle = \langle\mu_0^n,T_t^n\varphi\rangle < \infty,
    \]
    since $\mu_0^n$ is compact supported. Thus $M_t$ is a martingale and the result follows.
\end{proof}

Our next goal is to obtain the moment estimation of quantity $\sup_{0\leq s \leq t}\langle \mu_t^n,\varphi\rangle$ for some suitable $\varphi$. The direct operation on the $\mu_t^n$ suffers some technique issues on applying $\fH^n$ to the functions in the domain of operator $\fH$. It is non-trivial at all to see if the resulting functions $\fH^n\varphi^n$ with $\varphi^n\rightarrow\varphi \in \fD_\fH$ are uniformly bounded over $n$ at least in $L^\infty$. Thus, we turn to the estimation of sums of $\tilde{\mu}_t^n$ which governs $\mu_t^n$ and are much easier to manipulate. We will manipulate on $\mu_t^n$ as much as possible and only consider $\tilde{\mu}_t^n$ when there is some technical issue on $\mu_t^n$ that we can not solve. We will follow the argument of lemma 5.5.4 in \cite{dawson1992infinitely}.
\begin{lem}\label{lem.moment_est}
    Suppose $0 < \theta < \beta$ and $\mu_0$ is compactly supported. For any function $\varphi\in L^{\infty}(\ZZ_n^2,e(l))$ with some $l\in\RR$, we have the moments estimation uniformly over $n$
    \[
        \EE\left[\langle \mu_t^n,\varphi\rangle^{1+\theta}\right] \lesssim_{\mu_0,l,\theta} \|\varphi\|_{L^\infty(\ZZ_n^2,e(l))}^{1+\theta}.
    \]
    Furthermore, for any function $\varphi\in L^{\infty}(\ZZ_n^2,e(l))$ such that $\Delta^n\varphi \in L^{\infty}(\ZZ_n^2,e(l))$, we have the uniform estimation of finite moments of supreme over $n$.
    \begin{equation}
        \EE\left[\sup_{0\leq s\leq t}\langle \mu_s^n,\varphi\rangle^{1+\theta}\right] \lesssim_{\mu_0,l,\theta,\beta} \|\varphi\|_{L^\infty(\ZZ_n^2,e(l))} + \|\varphi\|_{L^\infty(\ZZ_n^2,e(l))}^2 + \|\Delta^n\varphi\|_{L^\infty(\ZZ_n^2,e(l))}^{1+\beta}.
    \end{equation}
\end{lem}
\begin{proof}
    \textbf{1.} Consider first a non-negative function $\varphi \in C^{1+2\kappa}(\ZZ_n^2,e(l))$ for some $l\in\RR$, we have 
    \begin{equation}\label{equ.moment_est}
        \begin{aligned}
            \EE[\langle \mu_s^n,\varphi\rangle^{1+\theta}]\lesssim  1 + (1+\theta)\int_1^\infty r^{1+\theta}\int_0^{\frac{2}{r}}[\EE[e^{-\langle \mu_s^n,u\varphi\rangle}] - 1 + \EE[\langle \mu_s^n,u\varphi\rangle]]dudr,
        \end{aligned}
    \end{equation}
    by lemma \ref{lem.auxiliary}. We calculate 
    \begin{equation*}
    \begin{aligned}
        &\EE[e^{-\langle \mu_s^n,u\varphi\rangle}] - 1 + \EE[\langle \mu_s^n,u\varphi\rangle] \\
        \lesssim& u^{1+\beta}\langle \mu_0^n,T_s^n(\varphi)\rangle^{1+\beta} + e^{-\langle \mu_0^n,U_s^n(u\varphi)\rangle} - e^{-\langle \mu_0^n,T_s^n(u\varphi)\rangle}\\
        \lesssim& u^{1+\beta}\langle \mu_0^n,T_s^n(\varphi)\rangle^{1+\beta} +\langle \mu_0^n,|U_s^n(u\varphi) - T_s^n(u\varphi)|\rangle.
    \end{aligned}
    \end{equation*}
    From the representation 
    \begin{equation*}
        U_t^n(u\varphi) = T_t^n\left(\frac{1-e^{-\epsilon u\varphi}}{\epsilon}\right) - \int_0^t T_{t-s}^n\left(\frac{2\xi_+^n\epsilon^\beta}{1+\beta}(U_s^n(u\varphi))^{1+\beta}\right)ds \leq uT_t^n\varphi,
    \end{equation*}
    we obtain 
    \begin{equation*}
        \begin{aligned}
            &|U_s^n(u\varphi) - T_s^n(u\varphi)| \\
            \lesssim & \left|T_t^n\left(\frac{1-e^{-\epsilon u\varphi}-\epsilon u\varphi}{\epsilon}\right)\right| + \left|\int_0^t T_{t-s}^n\left(\frac{2\xi_+^n\epsilon^\beta}{1+\beta}(U_s^n(u\varphi))^{1+\beta}\right)ds\right|\\
            \lesssim& \epsilon^\beta u^{1+\beta}|T_t^n(\varphi^{1+\beta})| + u^{1+\beta}\left|\int_0^tT_{t-s}^n\left(\frac{2\xi_+^n\epsilon^\beta}{1+\beta}T_s^n(\varphi)^{1+\beta}\right)ds\right|.
        \end{aligned}
    \end{equation*}
    Let 
    \begin{equation*}
        \begin{aligned}
            I^n_t(\varphi):=& \sup_{0\leq s\leq t}\langle\mu_0^n,T_s^n(\varphi)\rangle^{1+\beta} + \langle\mu_0^n,\epsilon^\beta T_s^n(\varphi^{1+\beta})\rangle \\
            &+ \left\langle\mu_0^n,\left|\int_0^sT_{s-r}^n\left(\frac{2\xi_+^n\epsilon^\beta}{1+\beta}T_r^n(\varphi)^{1+\beta}\right)dr\right| \right\rangle.
        \end{aligned}
    \end{equation*}
    Then we have the estimation 
    \begin{equation*}
        \begin{aligned}
            \EE[\langle\mu_s^n,\varphi\rangle^{1+\theta}] \lesssim 1 + (1+\theta)\int_1^\infty r^{1+\theta}\int_0^{\frac{2}{r}}u^{1+\beta}I_s^n(\varphi) dudr \lesssim 1 + \frac{1+\theta}{\beta-\theta}I_s^n(\varphi).
        \end{aligned}
    \end{equation*}
    The solution theory of PAM gives the bounds for any $\varphi \in C^{1+2\kappa}(\ZZ_n^2,e(l))$,
    \[
        \|T_s^n(\varphi)\|_{C_TC^{1-\kappa}(\ZZ^2_n,e(l+\cdot))} \lesssim \|\varphi\|_{C^{1+2\kappa}(\ZZ_n^2,e(l))},
    \]
    and by lemma 3.10 in \cite{martin2019paracontrolled}
    \begin{equation*}
        \begin{aligned}
        &\left\|\int_0^sT^n_{s-r}\left(\frac{2\xi_+^n\epsilon^\beta}{1+\beta}T_r^n(\varphi)^{1+\beta}\right)dr\right\|_{C_TC^{1-\kappa}(\ZZ_n^2,e(\hat{l}+\cdot))} \\
        \lesssim&  \left\|\frac{2\xi_+^n\epsilon^\beta}{1+\beta}\right\|_{C^{-\kappa}(\ZZ_n^2,p(b))}\left\|T_t^n(\varphi)^{1+\beta}\right\|_{C^{1-\kappa}(\ZZ_n^2,e(\hat{l}+\cdot))},
        \end{aligned}
    \end{equation*}
    where we choose $\hat{l} = ((1+\beta)l)\vee l$. It then suggests the bounds 
    \[
        I_t^n(\varphi) \lesssim \langle\mu_0^n,e(\hat{l}+t)\|\varphi\|\rangle^{1+\beta} + \langle\mu_0^n, e(\hat{l}+t)\epsilon^\beta\|\varphi\|^{1+\beta}\rangle,
    \]
    where $\|\varphi\|$ indicates $\|\varphi\|_{C^{1+2\kappa}(\ZZ_n^2,e(l))}$.
    Thus we have 
    \[
        \EE[\langle\mu_t^n,\varphi\rangle^{1+\theta}] \lesssim 1 +(\langle\mu_0^n,e(\hat{l}+t)\rangle^{1+\beta} + \langle\mu_0^n,e(\hat{l}+t)\rangle)\|\varphi\|^{1+\beta},
    \]
    for all non-negative $\varphi \in C^{1+2\kappa}(\ZZ_n^2,e(l))$ or $\varphi \in \fD^{1-\kappa}(\ZZ_n^2,e(l))$. In particular, if we choose $\varphi = e^{l\rho(x)}$, where $\rho(x)$ is a slight modification of $|x|^\sigma$ so that $\varphi \in C^{1+2\kappa}(\ZZ_n^2,e(l))$, we know the quantity 
    \[
        \EE\left[\langle \mu^n_t,e(l)\rangle^{1+\theta}\right]
    \]
    is finite uniformly over $n$ for all $l\in \RR$. It is then easy to see that if $\|\varphi\|_{L^\infty(\ZZ_n^2,e(l))}$ is finite, then we must have 
    \[
        \EE\left[|\langle \mu^n_t,\varphi\rangle|^{1+\theta}\right] \leq 2\|\varphi\|_{L^\infty(\ZZ_n^2,e(l))}^{1+\theta}\EE[\langle\mu^n_t,e(l)\rangle^{1+\theta}] < \infty,
    \]
    uniformly over $n$.
    
    \textbf{2. } Now on the moment estimation of supreme, we have to use branching random walks $\tilde{\mu}_t^{n,i}$. Notice first that by the Laplace functional of $\tilde{\mu}_t^n$, we have the same results as the first step for $\mu_t^n$. Furthermore, by the proof of lemma \ref{lem.martingale}, we know that 
    \[
        M_t := \langle\tilde{\mu}_t^n,\varphi\rangle - \langle\epsilon Z_0^n,\varphi\rangle - \int_0^t \langle\tilde{\mu}_s^n,\Delta^n\varphi\rangle ds
    \]
    is a martingale if $\Delta^n\varphi \in L^\infty(\ZZ_n^2,e(l))$ for some $l\in\RR$. Thus by Doob's inequality, we can have the control 
    \[
        \EE\left[\sup_{0\leq s\leq t}\langle\tilde{\mu}_s^n,\varphi\rangle^{1+\theta}\right] \lesssim \EE[|M_t|^{1+\theta}] + \EE\left[\int_0^t|\langle\tilde{\mu}_s^n,\Delta^n\varphi\rangle|^{1+\theta}ds\right] + \langle\mu_0,\varphi\rangle + \langle\mu_0,\varphi\rangle^2.
    \]
    Then by the coupling lemma \ref{lem.coupling}, we know
    \[
        \EE\left[\sup_{0\leq s\leq t}\langle\mu_s^n,\varphi\rangle^{1+\theta}\right] \lesssim_{N(\beta)} \sum_{i=1}^{N(\beta)}\EE\left[\sup_{0\leq s\leq t}\langle\tilde{\mu}_s^{n,i},\varphi\rangle^{1+\theta}\right],
    \]
    which gives the desired result.
\end{proof}

Now we are able to show the convergence of the particle system to the super process.
\begin{proof}[Proof of theorem \ref{thm.existence_laplace}]
    Suppose $\varphi \in C^{2}(\RR^2,e(l))$ for some $l\in\RR$. Consider $Y_t^n = \langle\mu_t^n,\varphi\rangle$. To prove the tightness of $(\mu_t^n)_n$, we apply Jakubowski's criterion. By lemma \ref{lem.moment_est}, we know 
    \[
        \langle\mu_t^n, 1 + |\cdot|^2\rangle 
    \]
    is uniformly bounded over $n$. Since for any $K$, the set $\{\mu:\langle\mu,1+|\cdot|^2\rangle < K\}$ is compact, we know compact containment property is satisfied by $\mu_t^n$. 
    
    Secondly, we apply Aldous's criterion. We have to prove 
    \[
    Y^n_{\tau_n+\delta_n} - Y^n_{\tau_n}\rightarrow0
    \]
in distribution, where $\tau_n$ are stopping times and $\delta_n$ decreasing to $0$. 
The moment bound shows that $\{Y^n_{\tau_n+\delta_n},Y^n_{\tau_n}\}$ are tight. It then suffices to consider the difference of Laplace transform
\begin{equation*}
    \EE\left[e^{-sY^n_{\tau_n+\delta_n}-tY^n_{\tau_n}}\right] - \EE\left[e^{-(s+t)Y_{\tau_n}}\right].
\end{equation*}
Taking the conditional expectation with respect to $\mu_{\tau_n}^n$, we have to obtain the formula of $\EE[e^{-\langle\mu_{\tau_n+\delta_n}^n,\varphi\rangle}|\mu_{\tau_n}^n]$. By lemma \ref{lem.laplace_equation_formula} and lemma \ref{lem.laplace_equi}, we know 
\[
    \EE[e^{-\langle\mu_{\tau_n+\delta_n}^n,\varphi\rangle}|Z_{\tau_n}^n = \delta_x] = 1 -\epsilon U_{\delta_n}^n(\varphi)(x).
\]
We then obtain that 
\[
    \EE[e^{-\langle\mu_{\tau_n+\delta_n}^n,\varphi\rangle}|\mu_{\tau_n}^n] = \EE[e^{-\langle Z_{\tau_n+\delta_n}^n,\epsilon\varphi\rangle}|Z_{\tau_n}^n] = e^{-\langle Z_{\tau_n}^n,-\log(1-\epsilon U_{\delta_n}^n(\varphi))\rangle}.
\]
Thus we have 
\begin{equation*}
    \begin{aligned}
    &\EE\left[e^{-sY^n_{\tau_n+\delta_n}-tY^n_{\tau_n}}\right] - \EE\left[e^{-(s+t)Y_{\tau_n}}\right] \\
    =& \quad\EE\left[e^{-\langle \mu_{\tau_n}^n,-\frac{1}{\epsilon}\log(1-\epsilon U^n_{\delta_n}(s\varphi)) - \langle\mu_{\tau_n}^n,t\varphi\rangle\rangle}\right] - \EE\left[e^{-\langle \mu_{\tau_n}^n,(t+s)\varphi\rangle}\right]\\
    \lesssim & \quad \EE\left[\left\langle\mu_{\tau_n}^n,\left|-\frac{1}{\epsilon}\log(1-\epsilon U^n_{\delta_n}(s\varphi)) - s\varphi\right|\right\rangle\right]\\
    \lesssim &\quad\EE\left[\sup_{0\leq s\leq t}\langle\mu_{s}^n,1\rangle\right]\left\|-\frac{1}{\epsilon}\log(1-\epsilon U^n_{\delta_n}(s\varphi)) - s\varphi\right\|_{L^\infty(\ZZ_n^2)}\\
    \lesssim &\quad \left\|-\frac{1}{\epsilon}\log(1-\epsilon U^n_{\delta_n}(s\varphi)) - U_{\delta_n}^n(s\varphi)\right\|_{L^{\infty}(\ZZ_n^2)} + \|U_{\delta_n}^n(s\varphi)  - s\varphi\|_{L^\infty(\ZZ_n^2)},
    \end{aligned} 
\end{equation*}
which goes to zero as $n\rightarrow\infty$ by estimation of SPDEs and lemma \ref{lem.auxiliary}. This tells us that we have the limit 
\[
    \left(Y_{\tau_n+\delta_n}^n , Y_{\tau_n}^n\right)\rightarrow (Y_\infty,Y_\infty)
\]
in distribution. Thus any linear combination of the above vectors also converges in distribution. In particular we have 
\[
    Y_{\tau_n+\delta_n}^n - Y_{\tau_n}^n \rightarrow 0
\]
in distribution. Thus we obtain the tightness of process $Y_t^n$. Since the family $C_c^\infty(\RR^2)$ separates points in $\fM_F(\RR^2)$, the Jakubowski's criterion gives the tightness of $\mu^n$ in $\DD([0,\infty),\fM_F(\RR^2))$. The discussion of SPDEs in section \ref{sec.variant} and the tightness of $Y_t$ tell us the limit process $\mu_t$ has the Laplace functional as stated in the theorem. For the $\rho < \beta$, the limit $\mu_t$ satisfies the Laplace functional 
\[
    \EE[e^{-\langle\mu_t,\varphi\rangle}] = e^{-\langle\mu_0,T_t(\varphi)\rangle} = e^{-\langle T_t\mu_0,\varphi\rangle},
\]
which implies that $\mu_t = T_t\mu_0$.
\end{proof}
\subsection{For the general $\varkappa >\frac{2}{1+\beta}$}
We do the mixing as in the paper \cite{perkowski2021rough}. 

\begin{definition}
    Fix a compact supported measure $\mu^n_0$ on $\ZZ_n^2$. Let $\epsilon = n^{-\frac{1}{\beta}}$ and $c > 0$. We define a $E^n-$valued stochastic process $Z^n(t)$ started at a Poisson random measure on $\ZZ_n^2$ with intensity $\frac{\mu^n_0}{\epsilon}$ with infinitesimal generator 
    \begin{equation}
        \fL^n F(\eta) = \sum_{x\in\ZZ_n^2} \eta_x\left[\Delta_x^n F(\eta) + |\xi^n|\sum_{k=0}^\infty p_kq_k(x)d_x^k F(\eta) + c|\xi^n|\sum_{k=0}^\infty p_kd_x^k F(\eta)\right]
    \end{equation}
    for any $F \in \fD(\fL^n)$ which consists of all $F\in C_b(E^n)$ such that the right-hand side of \eqref{equ.generator} is finite. Finally, set $\mu^n(t) = \epsilon Z^n(t)$ for all $t\in[0,\infty)$. 
\end{definition}
It follows that the branching rate of the above branching random walk is 
\[
    (1+c)|\xi^n|
\]
and the branching mechanism is given by 
\[
    \frac{g^n + cg}{1+c}
\]
Then this is clear that we can still find some controlled branching mechanism and perform the exact same procedure as for the case $\varkappa < \frac{2}{1+\beta}$. The final parameter of the $\beta-$rough super Brownian motion given by this mixture will be 
\[
    \varkappa = \frac{2\nu}{1+\beta} + \frac{2\nu c}{1+\beta}.
\]
This finishes the proof of theorem \ref{thm.existence_RSBM}.

\begin{remark}
    For $\alpha > \frac{1}{2}$, we can replace the Laplace with the generator of $\alpha-$stable process and obtain the convergence and existence of $(\alpha,d,\beta)-$super process.
\end{remark}

\section{Martingale problem}\label{sec.martingale_problem}
This section is devoted to proving the equivalence of three definitions of $\beta-$rough super Brownian motion. 

Let's start the first direction: from the Laplace functional to the martingale problem. We follow the method in book \cite{dawson1993measure}. Let's begin with a lemma in \cite{ethier2009markov} (corollary 3.3 in chapter 2)
\begin{lem}\label{lem.martingale_trans}
    Let $X,Y$ be real-valued, adapted right continuous processes. Suppose that for each $t$, $\inf_{s\leq t} X_s > 0$. Then 
    \[
        X_t - \int_0^t Y_sds
    \]
    is a local martingale if and only if 
    \[
        X_t exp\left(-\int_0^t \frac{Y_s}{X_s}ds\right)
    \]  
    is a local martingale.
\end{lem}
\begin{lem}
    Suppose $\mu$ satisfies the first definition in definition \ref{def.beta_RSBM}, then for all non-negative $\phi \in \fD_\fH$, the process
    \[
    N_t(\phi):=e^{-\langle\mu_t,\phi\rangle + \int_0^t\langle\mu_s,\fH\phi - \varkappa\phi^{1+\beta}\rangle ds},
\]
is a local $\fF-$ martingale.
\end{lem}
\begin{proof}
    Taking $B \in \fF_s$, we check 
\[
    \EE\left[(e^{-\langle\mu_t,\phi\rangle}-e^{-\langle\mu_s,\phi\rangle})\1_B \right] = \EE\left[ -\int_s^t\langle\mu_r,\fH\phi - \varkappa\phi^{1+\beta}\rangle e^{-\langle\mu_r,\phi\rangle}dr\1_B\right].
\]
It is sufficient to prove that for $t\geq s$,
\[
    \frac{d}{dt}\EE\left[e^{-\langle\mu_t,\phi\rangle}\1_B \right] = \EE\left[-\langle\mu_t,\fH\phi - \varkappa\phi^{1+\beta}\rangle e^{-\langle\mu_t,\phi\rangle}\1_B\right].
\]
By the definition, we have 
\begin{equation}
    \begin{aligned}
    \frac{d}{dt}\EE\left[e^{-\langle\mu_t,\phi\rangle}\1_B \right] &= \lim_{\epsilon\rightarrow 0}\frac{1}{\epsilon}\EE\left[\left(e^{-\langle\mu_{t+\epsilon},\phi\rangle} - e^{-\langle\mu_t,\phi\rangle}\right)\1_B\right]\\
    &= \lim_{\epsilon\rightarrow0}\frac{1}{\epsilon}\EE\left[\1_B e^{-\langle\mu_t,\phi\rangle}\left(e^{-\langle\mu_t,\epsilon\frac{U_\epsilon(\phi) -\phi}{\epsilon}\rangle}-1\right)\right].
    \end{aligned}
\end{equation}
Now, in order to apply the dominant convergence theorem, we have to verify 
\[
    \left\|\frac{U_{\epsilon}(\phi) - \phi}{\epsilon}\right\|_{L^\infty(\RR^2,e(l))}
\]
is uniformly bounded for some $l\in\RR$. By definition of $\fD_\fH$, there exists $\varphi \in C^{1+2\kappa}(\RR^2,e(l_0))$ such that $\phi = \int_0^t T_s\varphi ds$. Then 
\[
    U_t(\phi) = T_t(\phi) - \varkappa\int_0^t T_{t-s}(U_s(\phi))^{1+\beta}ds,
\]
which implies that
\begin{equation}
    \begin{aligned}
    \frac{1}{\epsilon}(U_\epsilon(\phi) - \phi) &= \frac{1}{\epsilon}\left(T_\epsilon(\phi)-\phi - \varkappa\int_0^\epsilon T_{\epsilon-s}(v_s(\phi))^{1+\beta}\right) \\
    &= \frac{1}{\epsilon}\left(\int_t^{t+\epsilon} T_s\varphi ds - \int_0^\epsilon T_s\varphi ds -\varkappa\int_0^\epsilon T_{\epsilon-s}(U_s(\phi))^{1+\beta}ds \right).
    \end{aligned}
\end{equation}
Thus we have 
\[
    \frac{1}{\epsilon}\|U_\epsilon(\phi) - \phi\|_{L^\infty(\RR^2,e(l))} \lesssim 1 + \|\varphi\|^{1+\beta}_{C^{1+2\kappa}(\RR^2,e(l_0))},
\]
for $l > 1 + ((1+\beta)l_0)\vee l_0\vee (l_0 + t)$ uniformly over $\epsilon$. In addition, we have by the continuity of $T_{\cdot}\varphi, U_{\cdot}\varphi$ that 
\[
    \lim_{\epsilon\rightarrow 0} \frac{1}{\epsilon}(U_\epsilon(\phi) - \phi) = T_t\varphi - \varphi - \varkappa\phi^{1+\beta} = \fH\phi - \varkappa\phi^{1+\beta}.
\]
Thus dominant convergence theorem applies and we know 
\[
    e^{-\langle\mu_t,\phi\rangle} + \int_0^t \langle\mu_r,\fH\phi-\varkappa\phi^{1+\beta}\rangle e^{-\langle\mu_r,\phi\rangle}dr
\]
is a local martingale. Lemma \ref{lem.martingale_trans} then shows that $N_t$ is a local martingale. 
\end{proof}
We are now ready to prove the theorem \ref{thm.equivalent_def}
\begin{proof}[Proof of theorem \ref{thm.equivalent_def}]
 \textbf{1.$\Rightarrow$ 3.} We first notice that since the Laplace functional uniquely determines the finite-dimensional law of the stochastic process, the $1+\theta$ moments estimation of $\sup_{0\leq s\leq t}\langle\mu_s,\phi\rangle$ for any $\phi\in L^\infty(\RR^2,e(l)),l\in\RR$ and $0 \leq \theta < \beta$ then follows by the first definition. We therefore have the $1+\theta$ boundedness of $L_t^\phi$ for any $\phi\in\fD_\fH$. 
 
 By repeating the argument in \cite{dawson1993measure}, we can obtain the martingale problem of $\mu$ from definition 1. Since the argument is quite long, we give only a sketch here and refer interested readers to chapter 6.1 in \cite{dawson1993measure}. First we choose a non-negative function $\phi\in\fD_\fH$. Let 
    
\[
    Y_t(\phi) = e^{-\langle\mu_t,\phi\rangle},\qquad H_t(\phi) = e^{-\int_0^t\langle\mu_s,\fH\phi-\varkappa\phi^{1+\beta}\rangle ds},
\]
then $Y_t(\phi) = H_t(\phi)N_t(\phi)$ is a strict positive semi-martingale. The It\^o's formula gives the representation 
\begin{equation}\label{equ.representation1}
    \begin{aligned}
    dY_t(\phi) &= H_t(\phi)dN_t(\phi) + N_{t-}(\phi)dH_t(\phi) \\
    &= H_t(\phi)dN_t(\phi) - Y_{t-}(\phi)\langle\mu_t,\fH\phi-\varkappa\phi^{1+\beta}\rangle dt
    \end{aligned}.
\end{equation}

Again by It\^o's formula, we know the process $\langle\mu_t,\phi\rangle = -\log Y_t(\phi)$ is again a semi-martingale. Now let $\fN(ds,d\nu)$ be adapted random point measure on $\RR_+\times \fM_F(\RR^2)$ given by $\sum \delta_{(s,\mu_s - \mu_{s-})}$, $\widehat{\fN}$ denotes its compensator and $\widetilde{\fN}$ the corresponding martingale measure. See appendix \ref{app.random} for more detail on random point measure.

The unique canonical decomposition of semi-martingale (lemma \ref{thm.semi_decomposition}) gives that 
\[
    \langle\mu_t,\phi\rangle = \langle\mu_0,\phi\rangle + B'_t(\phi) + M_t^C(\phi) + \widetilde{\fN}_t(\phi) + \fN_t(\phi),
\]
where $B_t'(\phi)$ is predictable, continuous and locally of bounded variation, $M_t^C(\phi)$ is a continuous local martingale with predictable increasing process $C_t(\phi)$ and 
\begin{equation*}
    \widetilde{\fN}_t(\phi) :=\int_0^t\int_{\fM_F^{\pm}(\RR^2)}\1_{|\langle\nu,\phi\rangle|\leq \|\phi\|}\langle\nu,\phi\rangle\widetilde{\fN}(ds,d\nu),
\end{equation*}
\begin{equation*}
    \fN_t(\phi) :=\int_0^t\int_{\fM_F^{\pm}(\RR^2)}\1_{|\langle\nu,\phi\rangle|> \|\phi\|}\langle\nu,\phi\rangle\fN(ds,d\nu).
\end{equation*}
The uniqueness of decomposition gives the relation
\[
    B'_t(\theta\phi) = \theta B'_t(\phi), \quad M_t^C(\theta\phi) = \theta M_t^C(\phi),\quad C_t(\theta\phi) = \theta^2C_t(\phi),
\]
for any $\theta > 0$. Applying It\^o's lemma \ref{lem.Ito} to $Y_t$ and obtain the second representation
\begin{equation}\label{equ.representation2}
\begin{aligned}
    dY_t(\phi) =& Y_{t-}(\phi)\left(-dB_t(\phi) + \frac{1}{2}dC_t(\phi) + \int_{\fM_F^{\pm}(\RR^2)}(e^{-\langle\nu,\phi\rangle} - 1+ \langle\nu,\phi\rangle)\widehat{\fN}(dt,d\nu)\right)\\
    &+ d(\text{local martingale}),
\end{aligned}
\end{equation}
where 
\[
    B_t(\phi) = B'_t(\phi) + \int_{\fM_F^{\pm}(\RR^2)}\langle\nu,\phi\rangle \1_{|\langle\nu,\phi\rangle|>\|\phi\|}\widehat{\fN}(dt,d\nu).
\]
Now since $Y_t$ is a special semi-martingale, the unique decomposition lemma \ref{lem.sp_decomposition} can be used to identify the predictable locally bounded variation part of two representations \eqref{equ.representation1} and \eqref{equ.representation2}, we have 
\[
    \langle\mu_t,\fH\phi - \varkappa\phi^{1+\beta}\rangle dt = dB_t(\phi) - \frac{1}{2}dC_t(\phi) - \int_{\fM_F^{\pm}(\RR^2)} (e^{-\langle\nu,\phi\rangle} - 1+ \langle\nu,\phi\rangle)\widehat{\fN}(dt,d\nu).
\]
Replacing $\phi$ with $\theta\phi$ and comparing coefficient before different power of $\theta$, we obtain $\int_0^t\langle\mu_s,\fH\phi\rangle ds= B_t(\phi), C_t(\phi) = 0$ and 
\begin{equation*}
    \begin{aligned}
        &\int_{\fM_F^{\pm}(\RR^2)} (e^{-\langle\nu,\phi\rangle} - 1+ \langle\nu,\phi\rangle)\widehat{\fN}(dt,d\nu) \\
        =& \langle\mu_t,\varkappa\phi^{1+\beta}\rangle dt = \int_0^\infty\frac{\varkappa\beta(1+\beta)}{\Gamma(1-\beta)u^{\beta+2}}\langle\mu_t,e^{-u\phi} - 1 + u\phi\rangle du,
    \end{aligned}
\end{equation*}
where $\Gamma(1-\beta)$ is the Gamma function. This shows that 
\begin{equation*}
    \widehat{\fN}(dt,d\nu) = \frac{\varkappa\beta(1+\beta)}{\Gamma(1-\beta)u^{\beta+2}}  dtdu\mu_t(dx)\delta_{u\delta_x}(d\nu).
\end{equation*}
Now for the general $\phi\in\fD_\fH$, we can decompose $\phi = \phi_+ - \phi_-$ and see that $\langle\mu_t,\phi\rangle$ is again a semi-martingale. Then by It\^o's formula, we have 
\begin{equation}
    \begin{aligned}
        \MM_t^f :=& f(\langle{\mu_t,\phi}\rangle) - f(\langle\mu_0,\phi\rangle) - \int_0^t f'(\langle\mu_s,\phi\rangle)\langle\mu_s,\fH\phi\rangle ds
        -\frac{\varkappa\beta(1+\beta)}{\Gamma(1-\beta)}\int_0^t\\
        &\left\langle\mu_t,\int_0^\infty\frac{1}{u^{\beta+2}}(f(\langle\mu_t,\phi\rangle + u\phi(\cdot)) - f(\langle\mu_t,\phi\rangle) - f'(\langle\mu_t,\phi\rangle)u\phi(\cdot))du\right\rangle dt
    \end{aligned}
\end{equation}
is a purely discontinuous local martingale for all $\phi \in \fD_\fH$. Now by comparing to the formula given by the compensator of $\langle\mu_t,\phi\rangle$ on $\RR$, we obtain the result.

\textbf{3.$\Rightarrow$ 2.} 
The proof is very similar to the proof in \cite{perkowski2021rough}. However, we have to be careful with non-square-integrable martingale. To start with, fix an $l<0$, from lemma \ref{lem.lower_bounds}, we know there exists  a non-negative function $\tilde{\phi} \in \fD_\fH$ such that $\frac{\tilde{\phi}}{e(l)}$ is strict larger than a small positive number. We choose a stopping time $T_n$ according to this $\langle\mu_s,\tilde{\phi}\rangle$ describing the $n-$th jump times with jump bigger than 1. Then the formula 
\[
    \EE\left[\sum_{0\leq r\leq T_n}1_{\langle\mu_r-\mu_{r-},\tilde{\phi}\rangle > 1}\right] = \EE\left[\int_0^{T_n} \int_0^\infty \1_{u>1}\langle\mu_r,\varkappa\phi^{1+\beta}\rangle\frac{\beta(1+\beta)}{\Gamma(1-\beta)u^{\beta+2}}dudr\right]
\]
gives that 
\[
    \EE\left[\int_0^{T_n}\langle\mu_r,\varkappa\tilde{\phi}^{1+\beta}\rangle dr\right] < \infty,
\]
which implies that $\EE\left[\int_0^{T_n}\langle\mu_r,e(l)\rangle\right] < \infty$ by the assumption of $\tilde{\phi}$. Then by the moments bounds of $L_t^{\tilde{\phi}}$ and doob's maximal inequality, we know that we actually have 
\begin{equation}\label{equ.supre_first_moment}
    \EE\left[\sup_{0\leq r\leq T_n}\langle\mu_{r},e(l)\rangle\right] < \infty.
\end{equation}
We will assume we have $\EE\left[\sup_{0\leq r\leq s}\langle\mu_{r},e(l)\rangle\right] < \infty$ and derive the local martingale property of 2. first.

Now consider a piecewise function $f:[0,t]\rightarrow\fD_\fH$ with weight $e(l-t+\cdot)$and $\varphi_0\in\fD_\fH$ with weight $e(l-t)$. Let $\varphi_t$ be the solution to the equation 
\[
    \partial_s\varphi_t(s) + \fH\varphi_t(s) = f(s), \quad s\in[0,t],\quad \varphi_t(t) = \varphi_0,
\]
which is given by $\varphi_t(s) = T_{t-s}\varphi_0 + \int_s^tT_{r-s}f(r)dr$. By assumption we have $\varphi_t(s) \in \fD_\fH$ for all $s\leq t$. For two time step $t_1 < t_2$ and a partition $\pi$ of $[t_1,t_2]$, we have 
\begin{equation}\label{equ.martingale_in}
    \begin{aligned}
        &\EE\left[\langle\mu_{t_2},\varphi_t(t_2) - \langle{\mu_{t_1},\varphi_{t}(t_1)}\rangle|\fF_{t_1}\right] \\
        =& \EE\left[\int_{t_1}^{t_2}\langle\mu_{s'},f(s)-\fH\varphi_t(s)\rangle+\langle \mu_s,\fH\varphi_t(s'')\rangle ds|\fF_{t_1}\right],
    \end{aligned}
\end{equation}
where $s'$ is the smallest partition point in $\pi$ bigger than $s$ and where $s''$ is the largest partition point in $\pi$ s maller than $s$. The right continuous property of $\mu_t$ in $\fM_F(\RR^2)$ shows that $\lim_{s'\rightarrow s ,s' >s}\langle\mu_{s'},(f(s) - \fH\varphi_t(s))\rangle = \langle\mu_{s},(f(s) - \fH\varphi_t(s))\rangle$, together with \eqref{equ.supre_first_moment}, we can send the mesh of partition of $[t_1,t_2]$ go to zero to get the limit of the first term of \eqref{equ.martingale_in}
\[
    \EE\left[\int_{t_1}^{t_2}\langle\mu_s,f(s)-\fH\varphi_t(s)\rangle ds|\fF_{t_1}\right].
\]
It remains to check 
\[
    \lim_{s'' \rightarrow s, s'' < s}\EE[\langle\mu_s,\fH\varphi_t(s'') - \fH\varphi_t(s)\rangle|\fF_{t_1}] \rightarrow 0.
\]
By the definition of $\varphi_t(s)$, we have 
\begin{equation*}
    \begin{aligned}
        &|\fH\varphi_t(s'') - \fH\varphi_t(s)| \\
        =& \left|T_{t-s''}\fH\varphi_0 - T_{t-s}\fH\varphi_0 + \int_s^tT_{r-s''}f(r)-T_{r-s}f(r)dr + \int_{s''}^sT_{r-s''}f(r)dr\right|\\
        &\lesssim (s-s'')^{\frac{1-\kappa}{2}}e(l).
    \end{aligned}
\end{equation*}
This gives the result and shows that $\MM_t^{\varphi_0,f}$ is a local martingale. And the density of $\fD_\fH$ gives the result. We will see it is actually a martingale when we show the direction 2.$\rightarrow$ 1. and 1. implies the moments estimation of $\langle\mu_t,\phi\rangle$.

From the compensator of $L_t^\phi$, we can obtain the compensator of the random point measure $\fN$ defined in the first step. Now since the purely discontinuous martingale is determined by its jump part and the jump of $\MM_t^{\varphi_0,f}$ is given by $\sum_{0\leq s\leq t} \langle\mu_s-\mu_{s-},\varphi_t(s)\rangle$. Then from the compensator of $\fN$, we know the compensator of $\MM_t^{\varphi_0,f}$ should be given by 
        \[
    \langle\mu_s,\varkappa\varphi_t(s)^{1+\beta}\rangle\frac{\beta(1+\beta)}{\Gamma(1-\beta)x^{\beta+2}}\1_{x>0}dsdx
\]

\textbf{2.$\Rightarrow$ 1.} We apply It\^o's formula to the semi-martingale $\langle\mu_s,\varphi_t(s)\rangle$ and obtain 
\begin{equation*}
\begin{aligned}
    &f(\langle\mu_s,\varphi_t(s)\rangle) \\
    =& f(\langle\mu_0,\varphi_t(0)) + \int_0^sf'(\langle\mu_r,\varphi_t(r)\rangle)\langle\mu_r,f(r)\rangle dr +\int_0^s\int_0^\infty \langle\mu_r,\varkappa\varphi_t(s)^{1+\beta}\rangle\\
    & \frac{\beta(1+\beta)}{\Gamma(1-\beta)u^{\beta+2}}(f(\langle\mu_{r-},\varphi_t(r)\rangle+u) - f(\langle\mu_{r-},\varphi_t(r)\rangle) - f'(\langle\mu_{r-},\varphi_t(r)\rangle)u)dudr\\
    &+ \text{local martingale}.
    \end{aligned}
\end{equation*}
Choose $f(x) = e^{-x}$ and $f(r) = -\varkappa (U_{t-s}(\varphi_0))^{1+\beta}$, then $\varphi_t(s) = U_{t-s}\varphi_0$ and 
\[
    e^{-\langle\mu_s,U_{t-s}\varphi_0\rangle} - e^{-\langle\mu_0,U_t\varphi_0\rangle}
\]
is a bounded local martingale. Thus it is a martingale, which gives the result.
\end{proof}
\section{Properties of RSBM}
In this section, we discuss the compact support property and the super exponentially persistence of the $\beta-$super Brownian motion. The compact support property for the rough super Brownian motion with $\beta = 1$ has been discussed in \cite{jin2023compact} and here we will give a sketch of how the method in \cite{jin2023compact} can be applied for the $0 < \beta < 1$. Apart from the assumption \ref{ass.environment}, we need as well the bounds
\[
    \sup_{x\in \RR^2} p(a)(x)^{-1}\int_{\RR^2}[(\fI(\fE^n\xi^n)(y) - \fI(\fE^n\xi^n)(x))(\fE^n\xi^n)(x) - c_n]\Psi^\delta(y-x)dy \lesssim \delta^{1-\kappa'}
\]
for any $a > 0,0<\delta<1$ with $\Psi^\delta$ defined in \cite{jin2023compact}.

Let $U_t^\phi\varphi$ be the solution to the equation 
    \begin{equation}\label{equ.inhomo_nonlinear_PAM_RSBM}
            \left\{
        \begin{array}{ll}
            \partial_t w_t = \fH w_t -  w_t^{1+\beta} + \phi, & \text{ in }(0,\infty]\times\RR^2,\\
            w_0 = \varphi, & \text{ on }\RR^2.
        \end{array}
    \right.
    \end{equation}
Then we have 
\begin{lem}
    For any function $\varphi\in C^{1+2\kappa}(\RR^2,e(l))$ and $\phi \in C^{-1+2\kappa}(\RR^2,e(l))$ for some $l\in\RR$, the process
    \[
        N_t(s) = e^{-\langle\mu_s,U_{t-s}^\phi\varphi\rangle - \int_0^s\langle\mu_r,\phi\rangle dr}
    \]
    is a $\fF-$martingale. In particular,
    \[
        \EE\left[e^{-\int_0^t \langle\mu_r,\phi\rangle dr}\right] = e^{-\langle\mu_0,U_t^\phi0\rangle}.
    \]
\end{lem}
\begin{proof}
    Let $f(x) = e^{-x}$ and $\varphi_t(s) = U_{t-s}^\phi\varphi$ in the It\^o's formula of $\langle\mu_s,\varphi_t(s)\rangle$ and then apply lemma \ref{lem.martingale_trans}, we can get the result.
\end{proof}
As discussed in \cite{jin2023compact}, let $\fP{n} = [-n,n]^2$ and $\fP{n}^m := \left(-n-\frac{1}{m},n+\frac{1}{m}\right)^2$. Define $\phi^n_m \in C^\infty_c$ such that
        \[
            \phi_n^m  
                \begin{cases}
                     = 0 & \text{ on $\fP{n}$}, \\
                     = m & \text{ on $(\fP{n}^m)^c$},\\
                     \in [0,m] & \text{ elsewhere.}
                \end{cases}
        \]
Since the limit 
\[
\lim_{m\rightarrow\infty}\EE\left[e^{-\int_0^t \langle\mu_r,\phi_n^m\rangle dr}\right]
\]
describes the probability that $\{\mu_s\}_{0\leq s\leq t}$ stays in the box $\fP{n}$ and the limit over $n$ describes the probability the compact support property holds. The compact support property holds if we show 
\[
    \lim_{n\rightarrow\infty}\lim_{m\rightarrow\infty}(U_t^{\phi_n^m}0)\eta_n = 0,
\]
uniformly over the support of $\mu_0$ for some cut-off function $\eta_n = 1$ on $\fP{n-1}$ and $0$ on $\fP{n}^c$. We here give a uniform interior estimation lemma on $U_t^{\phi_n^m} 0$ without proof. The similar result is proved in \cite{jin2023compact}. Also in \cite{moinat2020local,moinat2020space} for similar results of $\Phi_3^4$ model.
\begin{lem}\label{lem.local_est_weighted.nonlinear_PAM}
	    Let $n,m\in\NN$ and $a > 0$, we have 
        \begin{equation}
        \|(U_t^{\phi_n^m}0)\eta_n\|_{\fM_TC^{1-\kappa}(\RR^2,p(a))}\lesssim 1,
        \end{equation}
        where $'\lesssim'$ depends on time horizon $T$, $\|\xi\|_{C^{-1-\kappa}(\RR^2,p(b))}$ and $\|\fI\xi\diamond\xi\|_{C^{-2\kappa}(\RR^2,p(b))}$ for $b = \frac{4a}{\beta}$. Here the weight for the norm does not increase along time but remain $p(a)$. 
\end{lem}
\begin{proof}[Proof of theorem \ref{thm.RSBM_CSP}]
    The equation of $(U_t^{\phi_n^m}0)\eta_n$ is given by \[
        \partial_t (U_t^{\phi_n^m}0)\eta_n = \fH(U_t^{\phi_n^m}0)\eta_n - (U_t^{\phi_n^m}0)^{1+\beta}\eta_n + 2\nabla((U_t^{\phi_n^m}0)\nabla\eta_n) - (U_t^{\phi_n^m}0)\Delta\eta_n.
    \]
    Theorem \ref{thm.PAM} then gives a bound on $\|(U_t^{\phi_n^m}0)\eta_n\|_{\fL_T^{1-\kappa}}(\RR^2,e(l))$. Thus the compact embedding theorem tells us that, up to some subsequence, we have \[
        U_t:=\lim_{m\rightarrow\infty}\lim_{n\rightarrow\infty}(U_t^{\phi_n^m}0)\eta_n
    \]
    is the solution to the equation
    \[
        \partial_t U_t = \fH U_t - U_t^{1+\beta},\qquad U_0 = 0.
    \]
    Thus the uniqueness theorem tells us that $U_t = 0$. And we get the result.
\end{proof}
We now turn to the super-exponentially persistence of the generalized rough super Brownian motion. The super-exponentially persistence has been discussed in \cite{perkowski2021rough} using the approximation of super process in box of $\RR^2$. We here will use the result in \cite{perkowski2021rough}, together with the coupling method, to show the theorem \ref{thm.RSBM_SEP}.
\begin{proof}[Proof of theorem \ref{thm.RSBM_SEP}]
    We need to consider mixed super processes for the proof. Let $\bar{\mu}^n$ be the mixed particle systems with branching rate $(1+c)|\xi^n|$ and branching mechanism 
    \[
        \frac{h^n + ch}{1+c},
    \]
    where $h^n$ and $h$ are defined as $g^n$ and $g$ with $\beta = 1$. For $0 < \beta < 1$, we define a sequence of independent probability measures $\{\mu^{n,i}\}_{i\leq N(\beta)}$ corresponding to particle systems with branching rate $(1+c)|\xi^n|$ and branching mechanism 
    \[
        \frac{g^n + cg}{1+c}.
    \]
    For any $\varphi \in C_c^\infty(\RR^2)$, we now then need to compare the two quantities $\langle\bar{\mu}^n,\varphi\rangle$ and $\sum_{i=1}^{N(\beta)}\langle\mu^{n,i},\varphi\rangle$. 

    Let $Z_0$ be the random variable with generating function $\frac{h^n + ch}{1+c}$ and $Z_i$ be the random variable with generating function $\frac{g^n + cg}{1+c}$ for $i=1,\cdots,N(\beta)$, we need to prove 
    \[
        \PP[Z_0 \geq 2] \leq \PP\left[\sum_i^{N(\beta)} Z_i \geq 2\right],
    \]
    which is equivalent to showing that 
    \[
        \PP[Z_0 = 0] \geq \PP\left[\sum_i^{N(\beta)} Z_i = 0\right] = \PP\left[Z_1 = 0\right]^{N(\beta)},
    \]
    Since $c > 0$, we know $\PP[Z_0 = 0] >0$ and $\PP\left[Z_1 = 0\right] < 1$, the existence of $N(\beta)$ is guaranteed. Thus, there exists $N(\beta) > 0$ such that 
    \[
\langle\bar{\mu}^n,\varphi\rangle<\sum_{i=1}^{N(\beta)}\langle\mu^{n,i},\varphi\rangle. 
    \]
    Now consider event $A_0 = \{\lim_{t\rightarrow\infty}e^{-t\lambda}\langle\bar{\mu}^n,\varphi\rangle\}$, $A_i = \{\lim_{t\rightarrow\infty}e^{-t\lambda}\langle\bar{\mu}^n,\varphi\rangle\}$,then 
    \[
        A_0 \subset \bigcup_{i=1}^{N(\beta)}A_i.
    \]
    We also have $\PP[A_i] = \PP[A_1]$ for any $i = 1,\cdots, N(\beta)$. By the result in \cite{perkowski2021rough}, we know 
    \[
    \PP[A_0] > 0,
    \]
    Hence we must have $\PP[A_1] > 0$, which gives the result.
\end{proof}

\appendix
\section{Random point measure}\label{app.random}
We review some basics on the random point measure and purely discontinuous martingale in this  and next appendix. We refer reader to \cite{liptser2012theory} for more detail.

Let $(\Omega, \fF, \fF_t, \PP)$ be a probability space and $(E,\fE)$ be a Polish space with $\sigma-$algebra generated by topological sets. Denote $\RR_+ = [0,\infty)$. We give the definition of random measure on the space $\RR_+\times E$ here
\begin{definition}
    The family $\mu = \{\mu(\omega,dt,dx), \omega \in \Omega\}$ of $\sigma$ finite non-negative measure $\mu(\omega,\cdot)$ on $(\RR_+\times E,\fB(\RR_+)\otimes \fE)$, such that for each $A \in \fB(\RR_+)\otimes \fE$ the variable $\mu(\cdot,A)$ is $\fF$ measurable and 
    \[
        \mu(\omega,\{0\}\times E) = 0, \qquad \forall \omega \in \Omega,
    \]
    is called a random measure on $\RR_+\times E$. We call $\mu$ predictable if for any predictable random variable $X(\omega,s,x)$, the process 
    \[
        \int_0^t\int_E X(\omega,s,x)\mu(\omega,ds,dx)
    \]
    is predictable.
\end{definition}
We then give the definition of the compensator of a random measure $\mu$.
\begin{definition}
    A predictable random measure $\nu$ is called the compensator of the random measure $\mu$ if for any non-negative predictable random process $X(\omega,s,x)$, we have 
    \[
        \EE\left[\int_0^t\int_EX(s,x)\mu(dsdx)\right] = \EE\left[\int_0^t\int_EX(s,x)\nu(dsdx)\right]
    \]
\end{definition}
\begin{example}
    We will consider the random point measure $\mu$ associated to jumps of some process $X_t$ with value in $E$. The random point measure is defined via 
    \[
        \mu = \sum_{0\leq s\leq t} \delta_{s,X_{s} - X_{s-}}
    \]
    its compensator will play an important role in the It\^o's formula for a purely discontinuous martigale.
\end{example}
Here is another example of random measure on the space $E$ without time. 
\begin{example}
    A Poisson random measure $\Lambda$ on $E$ with intensity $\mu$ is a random measure with Laplace functional given by 
    \[
        \EE_\Lambda[e^{-\langle\cdot,\varphi\rangle}] = e^{-\int_{E}(1-e^{-\varphi(x)})\mu(dx)}
    \]
    for any non-negative measurable function $\varphi: E\rightarrow\RR$.
\end{example}
The Poisson random measure is a point measure whose formal definition is given by Poisson random variables describing the distribution of particle numbers in some set. We only need its Laplace characterization here which can also determine the Poisson random measure uniquely.

We also need a lemma to the Poisson cluster random measure.
\begin{lem}\label{lem.Poisson_cluster}
    Suppose we have a family of independent random measure $(\PP_x)_{x\in E}$ on the space $\tilde{E}$ and a Poisson random measure $\Lambda$ on $E$ with intensity $\mu$. Then the Laplace functional of the random measure $\PP_{\Lambda} : =\int_{E} \PP_x \Lambda(dx)$ is given by 
    \[
        \EE_{\Lambda}[e^{-\langle\cdot,\varphi\rangle}] = e^{-\int_{E}(1-\EE_x[e^{-\langle\cdot,\varphi\rangle}])\mu(dx)}
    \]
    for any non-negative measurable function $\varphi:\tilde{E}\rightarrow\RR$.
\end{lem}
\begin{proof}
    For a fixed realization of Poisson random measure $\Lambda$, the conditional Laplace functional of $\PP_{\Lambda}$ is given by 
    \[
        \prod_{x \in \Lambda} \EE_x\left[e^{-\langle\cdot,\varphi\rangle}\right] = e^{\int_{E}\log \EE_x[e^{-\langle\cdot,\varphi\rangle}]\Lambda(dx)} = e^{-\langle\Lambda,-\log \EE_x[e^{-\langle\cdot,\varphi\rangle}]\rangle}
    \]
    Now take the expectation with respect to Poisson random measure $\Lambda$ and use the Laplace functional of Poisson random measure, we obtain the result.
\end{proof}

\section{Canonical decomposition of a semi-martingale and It\^o's formula}

In this section, we say a stochastic process $A$ is of locally bounded variation if almost surely, the total variation of $A$, denoted by Var($A$), is finite on every finite interval. We also say $A$ is locally integrable if there exists a localization $T_n$ such that $\EE[Var(A)_{T_n}] < \infty$ for all $n\in\NN$.

It is well-known that a càdlàg stochastic process $X_t$ on $\RR$ is semi-martingale if it admits the decomposition
\[
    X_t = X_0 + M_t + A_t
\]
where $M_t$ is a local martingale and $A_t$ is a locally bounded variation process. The canonical decomposition of $X_t$ is given by 
\begin{theorem}\label{thm.semi_decomposition}
    Suppose $X_t$ is a semi-martingale and $a > 0$. Then $X_t$ admits a unique decomposition 
    \[
        X_t = X_0 + A_t^a + X_t^c + \int_0^t\int_{|x| \leq a}x d(\mu-\nu) + \int_0^t\int_{|x| > a} x d\mu
    \]
    where $X_t^c$ is the continuous martingale part of $X_t$, $A_t^a$ is a continuous predictable bounded variation process, $\mu$ is the jump measure of $X_t$ given by 
    \[
        \mu((0,t],C) = \sum_{0\leq s \leq t}\1_{\{X_s-X_{s-} \in C\}}, \quad C\subset \RR
    \]
    and $\nu$ is the compensator of $\mu$. 
\end{theorem}
Now consider a much more general case when the jump measure $\mu$ is defined on some Polish space $E$ and $\nu$ is again its compensator. Consider a semi-martingale in the form 
    \begin{equation}\label{equ.semimartingale}
        X_t = X_0 + A_t + X_t^c + \int_0^t\int_E g_1 d(\mu-\nu) + \int_0^t\int_E g_2 d\mu
    \end{equation}
    where $A_t$ is a continuous locally bounded variation process and $X_t^c$ is a continuous local martingale. For $g_1,g_2$, we also require $g_1g_2 =0$ and $\int_0^t\int_E |g_2|d\mu < \infty$ after some localization. The condition on $g_1$ is much more complicated and we refer the reader to the section 5, chapter 3 in \cite{liptser2012theory} for the definition.
\begin{lem}\label{lem.Ito}
    Let $f \in C^{1,2}(\RR_+\times \RR)$, then for semi-martingale $X_t$ with decomposition \eqref{equ.semimartingale}, we have the It\^o's formula
    \begin{equation}
    \begin{aligned}
        &f(t,X_t) - f(0,X_0) \\
        =& \int_0^t \partial_tf(s,X_s) ds + \int_0^t \partial_xf(s,X_s)dA_s + \frac{1}{2}\int_0^t \partial^2_{xx}f(s,X_s)d\langle X^c\rangle_s\\
        +&\int_{(0,t]}\int_E(f(s,X_{s-} + g_2(s,x)) - f(s,X_{s-}))\mu(ds,dx) \\
        +&\int_{(0,t]}\int_E(f(s,X_{s-}+g_1(s,x)) -f(s,X_{s-}) - \partial_x f(s,X_{s-})g_1(s,x))\nu(ds,dx)
    \end{aligned}
    \end{equation}
    when the last term is of locally integrable.
\end{lem}

We also need a unique decomposition of special semi-martingale whose supreme of jumps are of locally integrable. 
\begin{lem}\label{lem.sp_decomposition}
    Let $X_t$ be a special semi-martingale, then $X_t$ admits a unique decomposition 
    \[
        X_t = X_0 + A_t + M_t
    \]
    where $M_t$ is a local martingale and $A_t$ is of locally bounded variation.
\end{lem}

\section{Lemmas for Branching random walk}
We here give some useful lemma on branching random walk.
\begin{lem}\label{lem.laplace_equation_formula}
    For any branching random walk $Z_t$ on $\ZZ_n^2$ with branching mechanism $g(x,s) = \sum_{k=0}^\infty p_k(x)s^k$ and branching rate $a(x,ds)$, if 
    \[
        \sum_{k=0}^\infty kp_k(x) < K
    \]
    uniformly over $t,x$ with some finite constant $K$, then the branching random walk exists(no explosion, especially $\EE[\langle Z_t,1\rangle]$ is finite) and we have the expression for the Laplace functional of $Z_t$ for any test function $\varphi$ such that 
    \begin{equation}\label{equ.laplace_equation_formula}
    \begin{aligned}
        w_t(x) &:= \EE\left[e^{-\langle Z_t,\varphi\rangle}\right|Z_0 = \delta_x] \\
        &= \EE\left[e^{-\varphi(X_t) - a(0,t)} + \int_0^te^{-a(0,r)}g(r,X_r,w_{t-r}(X_r))a(X_r,dr)|X_0 = x\right]
    \end{aligned}
    \end{equation}
    where $X_t$ is the simple random walk on $\ZZ_n^2$ and 
    \[
        a(s,t) = \int_s^t a(X_r,dr)
    \]
\end{lem}
\begin{proof}
    The non-explosion follows by the corollary 1, page 111 in  \cite{athreya1972branching}. Let $\tau$ to be the first branching time of $Z$ in the time interval $[s,t]$ and $M(x)$ are independent random variables with the generator function $g(x,s)$. We define filtrations $\fF_t = \sigma(Z_s,s\leq t)$ and $\fG_t = \sigma(X_s,s\leq t)$. Then the random time $\tau$ is both stopping time of filtrations $\fF_t$ and $\fG_t$. Then we have 
    \begin{equation}
        \begin{aligned}
            w_{t}(x) &= \EE_{x}\left[e^{-\langle Z_t,\varphi\rangle}\1_{\tau \leq t} + e^{-\langle Z_t,\varphi\rangle}\1_{\tau > t}\right]\\
            &= \EE_{x}\left[\EE[w_{t-\tau}(X_\tau)^{M(X_{\tau})}|\fF_{\tau}]\1_{\tau \leq t} + e^{-\langle X_t,\varphi\rangle}\1_{\tau > t}\right]\\
            &=\EE_{x}\left[e^{-\langle X_t,\varphi\rangle - a(s,t)} + g(\tau,X_\tau,w_{t-\tau}(X_\tau))\1_{\tau\leq t}\right]
        \end{aligned}
    \end{equation}
    by \eqref{equ.branching_time} and the definition of moment generating function. Then by the law of $\tau$, we obtain the desired result.
\end{proof}
Here is a general lemma concerning the Feynman Kac formula and the mild solution to some equation.
\begin{lem}\label{lem.laplace_equi}
    For any measurable function $g:[0,\infty]\times \ZZ_n^2\times [0,1] \rightarrow [0,1]$, following the notation in lemma \ref{lem.laplace_equation_formula} ,the equation \eqref{equ.laplace_equation_formula} is equivalent to 
    \begin{equation}
        w_{t}(x) = \EE\left[e^{-\varphi(X_t)} + \int_0^t g(r,X_r,w_{t-r}(X_r))a(X_r,dr) - \int_0^t w_{t-r}(X_r)a(X_r,dr)\right]
    \end{equation}
\end{lem}
\begin{proof}
    See lemma 4.3.4 in \cite{dawson1993measure}, a simple version is given in the proof of lemma \ref{lem.variant_PAM} when we try to show the equivalence of Feynman Kac formulation and the mild formulation.
\end{proof}

\section{Auxiliary lemmas}
Here is an auxiliary lemma for our calculation of moment estimation.
\begin{lem}\label{lem.auxiliary}
    We have following inequalities:
    \begin{enumerate}
        \item $1 -x + \frac{x^2}{2} - e^{-x}\geq 0$, $x\geq 0$
        \item $\frac{x}{4} \leq 1 -x + \frac{x^2}{2} - e^{-x}$ for $x \geq 2$.
        \item $0\leq e^{-x} - 1+x\leq 2x^{1+\beta}$ for $x\geq 0, 0< \beta\leq 1$.
        \item $0 \leq -\frac{1}{\epsilon}\log(1-\epsilon x) - x\leq 2\epsilon x^2$ for $\epsilon >0, x \geq 0, \epsilon x \leq \frac{1}{2}$.
        \item For non-negative random variable $X$, we have $\EE[X^{1+\theta}] \leq (1+\theta)\delta^{1+\theta}+(1+\theta)\int_\delta^\infty r^\theta\PP[X\geq r]dr$, where  $\delta >0, \theta\geq -1$.
        \item $\PP[X\geq r] \leq  2r\int_0^{\frac{2}{r}}\EE\left[e^{-uX} - 1 + uX\right]du$, if $X\geq 0, \EE[X] < \infty$.
    \end{enumerate}
\end{lem}
\begin{proof}
    The first three inequalities are obvious and we omit the proofs. For the fourth inequality, we consider the expansion of $\log(1-x)$ when $0 \leq x < 1$, we have 
    \begin{equation*}
        \begin{aligned}
        -\frac{1}{\epsilon}\log(1-\epsilon x) - x &= \sum_{i=2}^\infty \frac{\epsilon^{i-1}x^i}{i} = \epsilon x^2\left(\frac{1}{2} + \sum_{i=1}^\infty\frac{\epsilon^ix^i}{i+2}\right) \\
        &\leq \epsilon x^2\left(\frac{1}{2} - \log(1-\epsilon x)\right) \leq 2\epsilon x^2
        \end{aligned}
    \end{equation*}
    The fifth inequality is from the basic probability theory. 
    \begin{equation*}
        \begin{aligned}
            \EE\left[X^{1+\theta}\right] &= \int_0^\infty \int_0^x (1+\theta)t^\theta dt \PP[dx]\\
            &= (1+\theta)\int_0^\infty \int_t^\infty t^\theta \PP[dx]dt \\
            &\leq (1+\theta)\left(\delta^{1+\theta} + \int_\delta^\infty t^\theta\PP[X\geq t]dt\right)
        \end{aligned}
    \end{equation*}
    The final inequality is given by 
    \begin{equation*}
        \begin{aligned}
            r\int_0^{\frac{2}{r}} \EE[e^{-uX} - 1 + uX] du &= r\int_0^{\frac{2}{r}}\int_0^\infty (e^{-ux} - 1 + ux) \PP[dx]du \\
            &= \int_0^{\infty} \frac{r}{x}\left(1 - e^{-\frac{2x}{r}} - \frac{2x}{r} + \frac{1}{2}\frac{(2x)^2}{r^2}\right)\PP[dx]\\
            &\geq \int_r^\infty \frac{x}{2r}\PP[dx] \geq \frac{1}{2}\PP[X\geq r]
        \end{aligned}
    \end{equation*}
    We use the second inequality for the third inequality.
\end{proof}

We also need to find a special function $\varphi_0\in\fD_\fH$ such that we have some lower bounds.
\begin{lem}\label{lem.lower_bounds}
    There exists positive function $\varphi_0\in\fD_{\fH}$ such that 
    \[
        \varphi_0(x) \geq Ce^{-l|x|^\sigma}
    \]
    for some constant $C$ and any $l>0$.
\end{lem}
\begin{proof}
    Consider $u_0 = 1$ and $A_tu_0$ for some horizon $t$. Then we have the representation 
    \[
        T^n_s u_0 = P^n_s u_0 + \int_0^s P^n_{s-r}((T^n_ru_0)\xi^n)ds 
    \]
    where $P_t^n$ is the discrete heat semigroup. Since $u_0 = 1$, we know $P^n_s u_0 = 1$ for all time $s$. Let 
    \[
        w^n_s = \int_0^s P^n_{s-r}((T_r^nu_0)\xi^n)ds
    \]
    we see by the solution theory of PAM that $w^n_s \in C^{\frac{1-\epsilon}{2}}L^\infty(\ZZ_n^d,e(t))$ uniformly over $n$. Thus there exists $C>0$ such that 
    \[
        |w_s^n(x)| \leq Cs^{\frac{1-\epsilon}{2}}e^{t|x|^\sigma}
    \]
    for each $x$, set $s(x) = \frac{e^{-\frac{2t}{1-\epsilon}|x|^\sigma}}{(2C)^{\frac{2}{1-\epsilon}}}$, then for any $0 \leq r \leq s(x)$, we have 
    \[
        |w_r^n(x)| \leq \frac{1}{2}
    \]
    Choosing $C$ large enough such that $s(x)\leq t$ for all $x \in \ZZ^2_n$. Then 
    \[
        A_t u_0(x) = \int_0^t T_su_0(x)ds \geq \frac{1}{2}\int_0^{s(x)}dr = \frac{s(x)}{2}
    \]
    This gives the result.
\end{proof}

    \bibliographystyle{amsalpha}
	\bibliography{RSBM}
\end{document}